\documentclass[a4paper,11 pt]{amsart}

\RequirePackage{amsmath, amssymb, amsthm, amsfonts, bm}
\usepackage{fullpage}

\makeatletter
\@namedef{subjclassname@1991}{$2020$ Mathematics Subject Classification}
\makeatother

\newtheorem{theo}{Theorem}[section]
\newtheorem{lem}[theo]{Lemma}
\newtheorem{prop}[theo]{Proposition}
\newtheorem{cor}[theo]{Corollary}

\theoremstyle{definition}
\newtheorem{defin}[theo]{Definition}

\newtheorem{exam}[theo]{Example}

\newtheorem{rem}[theo]{Remark}

\usepackage[all]{xy}
\usepackage{color}
\usepackage{hyperref}
\usepackage{placeins}

\newcommand{\void}[1]{}     

\newcommand{\HH}{\mathbb{H}}
\newcommand{\PP}{\mathbb{P}}

\newcommand{\EE}{\mathcal{E}}

\newcommand{\PSL}{{\mathrm{PSL}}}

\newcommand{\SL}{{\mathrm{SL}}}

\newcommand{\Aut}{{\mathrm{Aut}}}

\newcommand{\ZZ}{\mathbb{Z}}
\newcommand{\RR}{\mathbb{R}}
\newcommand{\QQ}{\mathbb{Q}}
\newcommand{\CC}{\mathbb{C}}

\newcommand{\tf}{\mathit{tf}}

\usepackage{tikz}
\usepackage{tkz-graph}

\title{On Elliptic K3 Surfaces and Dessins d'Enfants}
\author{Michael L\"onne and Matteo Penegini}
\address{Michael L\"onne\\  
Mathematik VIII, Universit\"at Bayreuth,
NWII, Universit\"atsstrasse 30, D-95447 Bayreuth,  Germany}
\email{michael.loenne@uni-bayreuth.de}
\address{Matteo Penegini\\
Universit\`a degli Studi di Genova\\
Dipartimento di Matematica - DIMA\\
Via Dodecaneso 35,
I-16146 Genova (GE) - 
Italy} 
\email{penegini@dima.unige.it}

\begin{document}


\begin{abstract} 
We classify subgroups of $\SL(2,\ZZ)$ up to conjugacy, which occur as monodromy groups of elliptically fibered K3 surfaces following a general strategy proposed by Bogomolov and Tschinkel.
The essential step is the factorisation of the functional invariant $j$ with second factor
a Belyi function of maximal possible degree and the classification of the corresponding
subgroups $\bar\Gamma$ of $\PSL(2,\ZZ)$ using dessins d'enfants.
\end{abstract}


\maketitle



\section{Introduction}

Elliptically fibred K3 surfaces have been the object of numerous articles
which exploit various aspects of their geometry and topology,
see \cite[Sec.11 \& 12]{SS} and the literature cited there.

The starting point of our work is the study of Bogomolov and Tschinkel 
\cite{BT03}
into the monodromy groups of elliptic fibrations. 
Recall that the monodromy group $\Gamma$ of an elliptically fibred surface
is a subgroup of $\SL(2,\ZZ)$. In fact there is a maximal open and dense subset
where the fibration restricts to a topological locally trivial torus bundle, and
$\Gamma$ is the monodromy group of that bundle up to conjugacy.

Bogomolov and Tschinkel propose a general strategy to find all monodromy
groups of elliptic surfaces, which we are going to make explicit in the case of
elliptic surfaces that are K3:
First one classifies all subgroups of $\PSL(2,\ZZ)$ that are quotients of 
monodromy groups, then one determines all subgroups of $\SL(2,\ZZ)$ with 
these quotients and finally one picks those, which are actually monodromy 
groups.

The first classification  relies on a bijection between subgroups of finite index
of $\PSL(2,\ZZ)$ and dessins d'enfants.

\begin{theo}\label{Theo_Main0}
The subgroups $\bar\Gamma$ of $\PSL(2,\ZZ)$ up to conjugacy, which occur
for elliptically fibered K3 surfaces are in bijection to $3228$
equivalence classes of planar dessins d'enfants.
\end{theo}

To get the list of conjugacy classes of monodromy groups in $\SL(2,\ZZ)$ one
has to add some finite information to the planar graphs.
This number is shown to depend only on the planar graph and how its symmetry
group acts on vertices and regions.
In fact we show:
\begin{itemize}
\item
$3411$ conjugacy classes of subgroups of $\SL(2,\ZZ)$ exist which occur
as monodromy groups of elliptically fibred K3 surfaces.
\item
$3153$ of these are in $1:1$ bijection with their quotients in $\PSL(2,\ZZ)$.
\item
$258$ map non-bijectively to $75$ classes of subgroups of $\PSL(2,\ZZ)$.
\end{itemize}

There is a similar study of Shimada, which counts instead all the isomorphism
types of root lattices, which occur for elliptically fibred K3 surfaces. There are
$3278$ such lattices according to \cite{ShK3} and a refined count on connected
families gives $3932$.
However, it is not true that the two counts give close numbers because of a
correspondence between monodromy groups and root lattices that is close
to a bijection. Rather as is shown in \cite{HL22} there are only relatively few
cases, where a monodromy group can be put into a correspondence with
a root lattice.

We will now discuss the contents of the paper in some more detail:

In Section \ref{sec_Prliminaries} we shall set up notation and terminology and we
present some preliminaries. In particular we recall the theory of modular subgroups and of elliptically fibered surfaces. 

The section \ref{sec_dessins} is devoted to presenting the theory of dessins d'enfantes and establishing the relations between these planar graphs and the modular subgroups of an elliptic fibration. It is worth noting that all the graphs for torsion free modular subgroups of index 6,12 and 18 are drawn in the section.   

In Section \ref{sec_Euler} using the theory of  $j$-modular curves we establish a simple formula that relates the Euler number of an elliptic fibration $\EE$ to the index of a modular subgroup $\bar{\Gamma}$ of $\PSL(2,\ZZ)$. This is the content of the Proposition \ref{prop_EulerEq}.

While the classification of torsion free modular subgroups up to conjugacy was already established see (e.g., \cite{SB21}) a classification of torsion modular subgroups was still missing. The content of Section \ref{sec_CTG} is to establish the number of torsion subgroups relating them to planar graph.

The last Section \ref{sec_CG} is devoted to the proof of Theorem \ref{Theo_Main0}.

\medskip

\textbf{Notation and conventions.} 
	
	We work over the field $\mathbb{C}$ of complex numbers.

	\medskip
	
\textbf{Acknowledgments}  We would like to thanks the organizing committee of the Conference "{\it Aspects of Algebraic Geometry}'' held in Cetraro in September 2023, in particular to give us the opportunity to work there in a stimulating environment. We would also like to thank Alice Garbagnati for useful conversations.   The second author was partially supported by GNSAGA-INdAM and PRIN 2020KKWT53 003 - Progetto: \emph{Curves, Ricci flat Varieties and their Interactions} and partially supported by the DIMA - Dipartimento di Eccellenza 2023-2027. 

\section{Preliminaries}\label{sec_Prliminaries}

The theory of elliptically fibered K3 surfaces is quite interesting because it connects many mathematical worlds that, at a first sight, have no connections. We are speaking about modular subgroups of finite index, branched covers of Riemann Surfaces and bipartite graph. Let us analyze these connections first.

\subsection{Modular Subgroups}

We recall that the \emph{Special Linear Group} with integer coefficients is the group 
\[
\SL(2,\ZZ):=\left\{ \begin{pmatrix} a & b \\
c & d \end{pmatrix}
\mid ad-bc=1, \, a,b,c,d \in \ZZ \right\}.
\]
Let $I=\textrm{diag}(1,1)\in \SL(2,\ZZ)$ then we can define the \emph{Modular Group}
\[
\PSL(2, \ZZ) := \SL(2,\ZZ)/ \pm I.
\]

The definitions above yield a short  exact sequence of groups
\begin{equation}\label{eq_SequenceSLPSL}
1 \longrightarrow \pm I \longrightarrow \SL(2, \ZZ) \stackrel{\xi}{\longrightarrow} \PSL(2, \ZZ) \longrightarrow 1.
\end{equation}

Consider the following classes of matrices in $\PSL(2,\ZZ)$ 
\[
S:=\xi\begin{pmatrix}0 & -1 \\
1 & 0\end{pmatrix}, \quad T:=\xi\begin{pmatrix}1 & 1 \\
0 & 1\end{pmatrix}.
\]
Then the elements $S$ and $ST$ have order respectively $2$ and $3$ in $\PSL(2,\ZZ)$. The following result is well known, see e.g. \cite[Theorem VII.2]{S73}.

\begin{theo} The group $\PSL(2, \ZZ)$ is generated by $S$ and $T$. Moreover, $\PSL(2,\ZZ)$ is isomorphic as a group to the free product $\ZZ/2 * \ZZ/3$. In particular, we have the following presentation
\[
\PSL(2,\ZZ) \cong \left\langle S,T\mid S^{2}=\left(ST\right)^{3}\right\rangle.
\]
\end{theo}

We are interested in studying the subgroups of $\SL(2,\ZZ)$ and $\PSL(2,\ZZ)$. 

\begin{defin} A \emph{Modular Subgroup} is a finite index subgroup of $\PSL(2, \ZZ)$. 
\end{defin}

Any finite index subgroup $\bar{\Gamma}$ of $\PSL(2,\ZZ)$ has a presentation with $g,r,h$ non negative integers and $m_i\in\{2,3\}$
as follows:

Generators:
\[
A_1, B_1,\ldots ,A_g, B_g, E_1,\ldots ,E_r, P_1,\ldots ,P_h.
\]
And relations:
\[
E^{m_1}_1 = \ldots = E^{m_r}_r = P_1 ...P_hE_1 ...E_r \prod^g_{i=1}[A_i,B_i] = 1.
\]
The generators $P_i$ are called \emph{parabolic}, the $E_i$ are called \emph{elliptic} and $A_i$ and $B_i$ are called \emph{hyperbolic}. From the presentation of the group $\bar{\Gamma}$, it is clear that $\bar{\Gamma}$ is \emph{torsion free}  if it has no elliptic elements. We call the positive integer $g$ the genus of the group. If the genus of $\bar{\Gamma}$ is zero, then $\bar{\Gamma}$  can be generated  by parabolic and elliptic elements only. We will work only with genus zero subgroups. Using Kuroshs theorem \cite{K34}, a torsion free subgroup of finite index $\bar{\Gamma}$ is a free group. 
Since from now $\bar{\Gamma}$ has genus zero, its rank is $h -1$, where $h$ is the number of parabolic generators of $\bar{\Gamma}$.

The modular subgroups are important because they are subgroups of the group of isometries of the hyperbolic plane. If we consider the upper half-plane model $\HH$ of hyperbolic plane geometry, then the group of all orientation-preserving isometries of $\HH$ consists of all M\"obius transformations of the form

\[
z\mapsto
\begin{pmatrix} a & b \\
c & d \end{pmatrix}z=Az= \frac{az + b}{cz + d}, \textrm{ with } A \in \SL(2,\RR)
\]
and induces an action of $\PSL(2,\RR)$.
A subgroup of finite index $n$ in $\PSL(2,\ZZ)$ has a fundamental domain in the upper
half-plane. We suppose its boundary to have $e_i$ inequivalent elliptic fixed point
vertices of order $i$ (for $i = 2$ and $3$) and $h$ inequivalent parabolic vertices, so that
the genus $g$ of the subgroup satisfies the Riemann-Hurwitz formula:
\begin{equation}\label{eq_RH}
g=1+\frac{n}{12}-\frac{e_2}{4}-\frac{e_3}{3}-\frac{h}{2}.
\end{equation}

\begin{defin}\label{def_type} 
The numerical datum 
\[
(n; g, h, e_2, e_3 )
\]
is called the \emph{type} of a modular subgroup.
\end{defin}

For geometric reasons, see Remark \ref{rem_Index24}, we will be interested in modular subgroups $\bar{\Gamma}$ of finite index equal to $6k$ with $k\in\{1,2,3,4\}$ and genus $0$. In this case, if $\bar{\Gamma}$  is torsion free, by \eqref{eq_RH} we have 
\[
n=6(h-2).
\]
Therefore, for $n=6k$ the number of the parabolic generators of $\bar{\Gamma}$ is $h=k+2$.

\begin{defin}\label{rem_TTT} 
The exact sequence \eqref{eq_SequenceSLPSL} induces an exact sequence of subgroups
\begin{equation}\label{eq_seqGG}
1 \longrightarrow \ZZ/2 \longrightarrow \xi^{-1}(\bar{\Gamma}) \stackrel{\xi}{\longrightarrow} \bar{\Gamma} \longrightarrow 1.
\end{equation}
We call a \emph{lift} of $\bar{\Gamma}$ any subgroup $\Gamma \subset \xi^{-1}(\bar
{\Gamma})$ that surjects onto $\bar{\Gamma}$.
\end{defin}

To understand all lifts $\Gamma$ of any given modular subgroup $\bar{\Gamma}$
note that either $\Gamma = \xi^{-1}(\bar{\Gamma})$ or $\Gamma$ is the image
of a homomorphism splitting the exact sequence \eqref{eq_seqGG}.

\begin{lem}\label{lem_split} The group $\bar{\Gamma}$ does not contain elements of order 2 if and only if the exact sequence \eqref{eq_seqGG} splits.
\end{lem}

\begin{proof} One implication is proven in  \cite[Lemma 4.1]{BT03}. Following their argument and using Definition \ref{rem_TTT} we have the following. If $\bar{\Gamma}$ has no element of order $2$ then it has a presentation as a free product of copies of $\ZZ$ and $\ZZ/3$. We lift the generators of these free generating subgroups to elements of the same order in $\SL(2,\ZZ)$ and we obtain a subgroup $\Gamma$ which projects isomorphically onto $\bar{\Gamma}$. This yields the desired splitting.

It remains to prove that a splitting of  \eqref{eq_seqGG} does not exists if $\bar{\Gamma}$ contains a $2$-torsion element. This is clear, because the preimage of a $2$-torsion is a $4$-torsion element since $-I$ is the only element of order $2$ in $\SL(2,\ZZ)$. 
\end{proof}

\begin{lem}\label{lem_with2tors} If $\bar{\Gamma}$ does  contain elements of order 2 then the map $\xi_{|_{\Gamma}}$ in the exact sequence \eqref{eq_seqGG} is two to one. Moreover, the lift $\Gamma\subset \SL(2,\ZZ)$ of $\bar{\Gamma}$ is unique.
\end{lem}

\begin{proof} The first assertion follows from the exact sequence \eqref{eq_seqGG}. The uniqueness of the lifts follows from the fact that $\Gamma$ must contain $-I$.
\end{proof}

 If we are not in the situation of Lemma \ref{lem_with2tors}, then there may be several lifts of $\bar{\Gamma}$.

\subsection{Elliptically fibered K3 surfaces}
As already mentioned modular subgroups have many geometric interpretations, we shall develop one of these here.
A K3 surface $S$ is said to be  \emph{elliptically fibered} if there is a surjective map
\[
f\colon S \longrightarrow \PP^1
\]
whose generic fibre is a genus 1 curve. The fibration  $f\colon S \longrightarrow \PP^1$ is a \emph{Jacobian fibration} if it has a section $s$. Moreover, as custom, we shall denote a Jacobian fibration by $f\colon \mathcal{E} \longrightarrow \PP^1$ (or simply by $\mathcal{E}$) where we think of $\mathcal{E}$ as an elliptic curve over the function field $\CC(t)$. Notice that  we suppose that all the fibrations considered in this text are \emph{relatively minimal}. 

We recall the definition of the \emph{monodromy group} $\Gamma$ of $\EE$. For a detailed construction of $\Gamma$ see e.g. \cite[Section 1]{BT03}, here we sketch only the main idea.  The complement  of the union of singular fibres 
of a Jacobian fibration $\EE$ is topologically a torus bundle over the base punctured
at the critical values of the fibration.  Letting $B$ the set of points in $\PP^1$ over which $\EE$ has a singular fibre, we have that  $\EE |_{\PP^1 \setminus B} = f^{-1}(\PP^1 \setminus B) \to \PP^1 \setminus B$ is even a differentiable torus bundle. Hence the first homology groups  $H_1(F, \, \ZZ)$ of the fibers $F$ form a locally constant sheaf on $\PP^1 \setminus B$. This sheaf extends canonically to a sheaf $\mathcal{G}$ on $\PP^1$, the \emph{homological invariant} of $\EE$ introduced by Kodaira.

The image of the associated monodromy representation
\[
\rho \colon \pi_1(\PP^1\setminus B) \longrightarrow \Aut (H_1(F, \, \ZZ))
\]
in the automorphisms of the first homology of a fibre is orientation preserving. Upon the
choice of a basis, it becomes the monodromy group $\Gamma:=\Gamma(\EE)$, a subgroup of $\SL(2, \ZZ)$ (the group of orientation preserving automorphism of a torus), well-defined
up to conjugacy.  Thus, we get a representation called the \emph{monodromy homomorphism}
\[
\rho_{\EE}\colon \pi_1(\PP^1\setminus B) \longrightarrow \Gamma \subset \SL(2,\ZZ).
\]
For what we just said this homomorphism is defined only modulo conjugacy in $\SL(2,\ZZ)$.

The group $\Gamma(\EE)$, as we have seen, determines a group $\bar{\Gamma}( \EE) \subset \PSL(2, \ZZ)$ which we will call the \emph{modular monodromy group}. Therefore, we have associated to a Jacobian fibration a modular subgroup. We aim at classifying or at least counting the number of such modular subgroups associated to Jacobian fibrations of K3 surfaces up to conjugacy. As we shall see, this will be the decisive step to determine the possible monodromy groups of Jacobian elliptic K3 surfaces.
\medskip

We will also heavily rely on the Euler number of a K3 surface to be equal to 24.
In case of an elliptic fibration the Euler number can be calculated by looking only at the contribution of each singular fibers. The possible singular fibers and their contributions to the Euler number were classified by Kodaira and the classification can be summarized by the following table (see also \cite[Table V.6]{BHPV}). 

\begin{table}[h]
\centering
\begin{tabular}{|c|c|c|c|c|} \hline
 fibre type & 
$ADE$-type  & 
Euler number  & local monodromy & local $j$-expansion
\\ \hline
 & & & &\\[-4mm] \hline
 & & & &\\[-4mm]
 $I_0$ & -- &
$0$ & $(\begin{smallmatrix} 1 & 0 \\ 0 & 1 \end{smallmatrix})$ 
& 
\begin{tabular}{c} $j=s^{3k}$, $j=1+s^{2k}$\\ or $j\neq0,1$ \end{tabular}
\\[-4.2mm] &&&&
\\ \hline  & & & &\\[-4mm]
 $I_1$ & -- &
$1$ & $(\begin{smallmatrix} 1 & 1 \\ 0 & 1 \end{smallmatrix})$ & pole of order $1$
\\[-4.2mm] &&&&
\\ \hline & & & &\\[-4mm]
 $I_b\:\: (b\geq 2)$ &
$A_{b-1}$ & $b$ 
& $(\begin{smallmatrix} 1 & b \\ 0 & 1 \end{smallmatrix})$ & pole of order $b$
\\[-4.2mm] &&&&
\\ \hline & & & &\\[-4mm]
 $I^*_0$ &
$D_{4}$ & $6$ 
& $(\begin{smallmatrix} -1 & 0 \\ 0 & -1 \end{smallmatrix})$ & same as in first case
\\[-4.2mm] &&&&
\\ \hline & & & &\\[-4mm]
 $I^*_b \:\:  (b\geq 1)$ &
$D_{4+b}$ & $6+b$ 
& $(\begin{smallmatrix} -1 & -b \\ 0 & -1 \end{smallmatrix})$ & pole of order $b$
\\[-4.2mm] &&&&
\\ \hline & & & &\\[-4mm]
 $II$ & -- &  $2$ 
& $(\begin{smallmatrix} 1 & 1 \\ -1 & 0 \end{smallmatrix})$ & $j = s^{3k+1}$
\\[-4.2mm] &&&&
\\ \hline & & & &\\[-4mm]
 $III$ & $A\sb 1$ & $3$ 
& $(\begin{smallmatrix} 0 & 1 \\ -1 & 0 \end{smallmatrix})$ & $j = 1+s^{2k+1}$
\\[-4.2mm] &&&&
\\ \hline & & & &\\[-4mm]
 $IV$ & $A_2$ & $4$ 
& $(\begin{smallmatrix} 0 & 1 \\ -1 & -1 \end{smallmatrix})$ & $j = s^{3k+2}$
\\[-4.2mm] &&&&
\\ \hline & & & &\\[-4mm]
 $IV^*$ &
$E\sb 6$ & $8$ 
& $(\begin{smallmatrix} -1 & -1 \\ 1 & 0 \end{smallmatrix})$ & $j = s^{3k+1}$
\\[-4.2mm] &&&&
\\ \hline & & & &\\[-4mm]
 $III^*$ &
$E\sb 7$ & $9$ 
& $(\begin{smallmatrix} 0 & -1 \\ 1 & 0 \end{smallmatrix})$ & $j = 1+s^{2k+1}$
\\[-4.2mm] &&&&
\\ \hline & & & &\\[-4mm]
 $II^*$ &
$E\sb 8$ & $10$ 
& $(\begin{smallmatrix} 0 & -1 \\ 1 & 1 \end{smallmatrix})$ & $j = s^{3k+2}$
 \\[-4.2mm] &&&&
\\ \hline 
\end{tabular}
\caption{Fibre types of elliptic surfaces} \label{Table01}
\label{table:kodaira}
\end{table}


\subsection{The $j$-modular curves}\label{sec_jmod}

An essential tool for the classification of Jacobian fibrations $\mathcal{E}$ is the $j$-function associated to  $\mathcal{E}$. Let $\HH$ the upper half plane and set
\[
\bar{\HH}=\HH \cup \QQ \cup \{ \infty\}
\]
the completed upper half plane. The quotient 
\[
X(1)= \bar{\HH}/\PSL(2,\ZZ) \cong \PP^1
\]
is called the \emph{modular curve of level} $1$, namely the compactification of the $j$-line.

For a Jacobian fibration $f\colon \EE \to \PP^1$, using  the $j$-function of elliptic curves we can define a rational function
\[
\begin{split}
j(\EE) \colon \PP^1 \longrightarrow & X(1) \cong \PP^1, \\
b \mapsto & j(f^{-1}(b)).
\end{split}
\]
We work with coordinates on $X(1)$ such that the Gauss curve has image $1$ and the
Eisenstein curve has image $0$.

Following Kodaira we call $j(\EE)$ the \emph{functional invariant} of the elliptic fibration. Notice that $j(\EE)$ is a holomorphic function on $\PP^1 \setminus B$ and it is meromorphic on the whole $\PP^1$. 

Recall from the previous section that we have a monodromy homomorphism $\rho_{\EE}$ and composing it with $\xi$ we get a commutative diagram
\[
\xymatrix{
  \pi_1(\PP^1\setminus B)  \ar[dr]_{\rho'} \ar[rr]^{\rho_{\EE}} & & \SL(2,\ZZ)  \ar[dl]^{\xi}\\
  & \PSL(2,\ZZ), &
   }
\]

where $\rho'=\xi\circ \rho_\EE$.

In this situation we shall call $\rho_\EE$ a \emph{lift} of $\rho'$.
Consider a geometric loop $\gamma$ around a point $b \in B$. Then $\gamma$ determines a class of matrices in $\PSL(2, \ZZ)$.  A \emph{minimal lift} $\rho(\gamma)$ of $\rho'(\gamma)$ is defined to be a choice of a representative matrix in $\SL(2, \ZZ)$ such that $\xi(\rho(\gamma))=\rho'(\gamma)$ and $-\rho(\gamma)$ is the monodromy of a $*$-fiber.  Here the adjective minimal means that the contribution to the Euler number to the fibre over $b$ is the smallest.

Moreover, the \emph{homological invariant} sheaf $\mathcal{G}$ on $\PP^1$ associated to $\rho_{\EE}$ \emph{belongs to $j(\EE)$ in the sense of Kodaira}, if a unique Jacobian fibration $f\colon \EE \to \PP^1$ can be constructed for which the homological invariant is $\mathcal{G}$ and the functional invariant is $j(\EE)$, see e.g., \cite{Nori}.

 As said, we are interested in classifying monodromy groups for Jacobian fibrations. In the light of the last observation, we have to investigate both the functional and the homological invariant. Notice that the degree of the $j(\EE)$ is bounded. This implies a bound on the index of the monodromy group $\Gamma \subset \SL(2,\ZZ)$. Only  a finite number of possible monodromy groups $\Gamma$ and only few homological invariants can occur if we fix the image of $\Gamma$ in $\PSL(2,\ZZ)$.

Furthermore, to achieve our goal, we will make great use of the following valuable general lemma.
 
 \begin{lem}\cite[Lemma 2.3]{BT03}\label{lem_factor} Let $C$ a smooth curve. Let $j\colon C \to \PP^1$ be any non constant rational map. Then there exists a unique subgroup of finite index $\bar{\Gamma} \subset \PSL(2,\ZZ)$ such that
 \begin{enumerate}
 \item  $j$ decomposes as
\begin{equation*}\label{eq_jFactor}
j\colon C \longrightarrow \bar{\HH}/\bar{\Gamma} \longrightarrow \bar{\HH}/\PSL(2,\ZZ)\cong \PP^1.
\end{equation*}
\item for every $\bar{\Gamma}' \subset \PSL(2,\ZZ)$ with a factorization as in $(1)$ there exists an element $g\in \PSL(2,\ZZ)$ such that $g\bar{\Gamma}g^{-1} \subset \bar{\Gamma}'$.
\end{enumerate}
\end{lem}

Now suppose that $\EE$  has monodromy group $\Gamma \subset \SL(2,\ZZ)$,  we shall denote its image in $\PSL(2,\ZZ)$ by $\bar{\Gamma}$. We define the \emph{$j$-modular curve associated to $\bar{\Gamma}$}
\[
X(\bar{\Gamma}):=\bar{\HH}/\bar{\Gamma}.
\]

The genus of the modular curve can be calculated with the Riemann-Hurwitz formula in the following way.
Let $n$ the index of $\bar{\Gamma}$ in $\PSL(2,\ZZ)$, $e_2$ and $e_3$ the number of inequivalent elliptic fixed points of order $2$ and $3$ respectively, and let $h$ the number of the inequivalent cusps. Then the Riemann--Hurwitz formula yields
\begin{equation}\label{eq_RH2}
g\left(X(\bar{\Gamma})\right)=1+\frac{n}{12}-\frac{e_2}{4}-\frac{e_3}{3}-\frac{h}{2}.
\end{equation}

Since $B$ is rational in our case, then by L\"uroth theorem $X(\bar{\Gamma})$ is again isomorphic to $\PP^1$, hence 
\[1+\frac{n}{12}-\frac{e_2}{4}-\frac{e_3}{3}-\frac{h}{2}=0. 
\]
By Lemma \ref{lem_factor} the function $j(\EE)$ has a unique factorization 
\[
j(\EE) = j_{\bar{\Gamma}}\circ j_{\EE}
\]
up to deck transformations of $j_{\bar{\Gamma}}$, where
\[
j_{\EE}\colon \PP^1 \longrightarrow X(\bar{\Gamma})\cong\PP^1  
\]
and
\[
j_{\bar{\Gamma}}\colon X(\bar{\Gamma})\cong\PP^1 \longrightarrow X(1) \cong \PP^1
\]

This last covering is branched only on $0$, $1$ and $\infty$. The branch multiplicity is arbitrary over $\infty$, $1$ or $3$ over $0$ and $1$ or $2$ over $1$. 

The above decomposition shows that
\begin{equation}\label{eq_jDecomp}
\deg(j(\EE))=\deg(j_{\bar{\Gamma}}) \cdot \deg(j_{\EE}).
\end{equation}

We shall call the \emph{j-factor} of the surface $\EE$ the degree $\deg(j_{\EE})$. We will denote it often with the letter $l$.

This discussion tells us that in order to classify monodromies of Jacobian fibration  we have to investigate the subgroups of $\PSL(2,\ZZ)$ by means of the following three sets:
\begin{enumerate}
\item subgroups $\bar{\Gamma}$ of   $\PSL(2,\ZZ)$ of index $n$ up to conjugacy in $\PSL(2,\ZZ)$.
\item branched covers $j\colon C \to X(1)$ of degree $n$ up to equivalence of cover, branched only over $0$, $1$
and $\infty$ with multiplicity $1$ or $3$ over $0$ and  with multiplicity $1$ or $2$ over $1$.
 \item monodromy homomorphisms $\mu\colon \pi_1(\CC \setminus \{0,1\}) \longrightarrow \mathfrak{S}_n$ up to conjugacy in $\mathfrak{S}_n$, such that simple loops around $0$, respectively $1$ maps to elements of order $1$ or $3$, resp. $1$ or $2$, and the image acts transitively. 
\end{enumerate}

It is worthwhile to see that these sets are in bijection with each other in a natural way (see also \cite{HL22}). 

First of all notice that $\pi_1(\CC \setminus \{0,1\})$ is freely generated by simple geometric loops around $0$ and $1$, while $\PSL(2,\ZZ)$ is the free product $\ZZ/2 * \ZZ/3$.

\begin{description}
\item[(1)$\to$(2)]
Given $\bar{\Gamma}$, associate the branched cover 
$j_{\bar{\Gamma}}\colon X(\bar{\Gamma}) \longrightarrow X(1)$.
\item[(1)$\to$(3)]
Given $\bar{\Gamma}$, associate left multiplication 
$\PSL(2,\ZZ) \to \textrm{Perm}(\PSL(2,\ZZ)/\bar{\Gamma})\cong \mathfrak{S}_n$ on cosets
and compose with $\pi_1(\CC\setminus\{0,1\}) \to \PSL(2,\ZZ)$ to get $\mu$.
\item[(2)$\to$(3)]
Given $j:C\longrightarrow \PP^1$, restrict to $\CC\setminus\{0,1\}$, which is a topological
cover of degree $n$ and associate the representation
$\mu:\pi_1(\CC\setminus\{0,1\}) \to \mathfrak{S}_n$ by permutations of a fibre.
\item[(3)$\to$(1)]
Given $\mu$, note that by our assumptions this factors through a homomorphism $\bar{\mu}\colon \PSL(2,\ZZ)\longrightarrow  S_n$ and
associate the stabilizer subgroup $\bar{\Gamma}\subset\PSL(2,\ZZ)$ of $1$.
\item[(3)$\to$(2)]
Given $\mu$, associate the connected topological cover over $\CC\setminus\{0,1\}$
of degree $n$ that extends to a branched cover $j:C\to \PP^1$ by Riemann existence.
\end{description}

\begin{prop}\label{prop_FactJ}\cite[Lemma 6.3]{HL22} Given the holomorphic map $j(\EE)\colon \PP^1 \longrightarrow X(1)$, a factorization $j_2\circ j_1$ is equivalent to $j_{\bar{\Gamma}} \circ j_\EE$ if and only if
\begin{enumerate}
\item $j_2$ is branched only over $0,1,\infty$ with multiplicity 1,3 over 0 and 1,2 over 1. 
\item $j_1$ has no proper left factor $j'$, so that $j_2\circ j'$ has the property above.
\end{enumerate}
\end{prop}

We can associate with $j_{\bar{\Gamma}}$ three conjugacy classes $C_0,C_1$ and $C_{\infty}$ corresponding to the branch points $0,1,\infty$. This means that we give the three conjugacy classes as a $3$-tuple of partitions of $\deg( j_{\bar{\Gamma}} )$. The degree of $j_{\bar{\Gamma}}$  and the numbers $e_3$ and $e_2$ of fixed points elements in $C_0$ and $C_1$, respectively, are related by the following congruences:
\begin{equation*}
\deg j_{\bar{\Gamma}}  \equiv_2 e_2, \qquad \deg j_{\bar{\Gamma}}  \equiv_3 e_3.
\end{equation*}
This follows since the difference of the two numbers in questions is the sum of cycle lengths of transpositions and $3$-cycles, respectively. Note that for the subgroup $\bar{\Gamma} \subset \PSL(2,\ZZ)$ corresponding to $j_{\bar{\Gamma}}$ the numbers $e_2$ and $e_3$ count the non ramified preimages of $1$ and $0$ under $j_{\bar{\Gamma}} $ which we call \emph{$2$-torsion points} and \emph{$3$-torsion points} respectively. Notice that $e_2$ and $e_3$ in the equations \eqref{eq_RH} and \eqref{eq_RH2} are the same numbers.

\section{Dessins D'Enfants and Enumeration}\label{sec_dessins}

Proceeding from the concluding observations of the previous chapter we will
arrive at a satisfactory visual classification of the modular subgroups we are
interested in.
\

Grothendieck \cite{Gr} in his esquisse d'un programme proposed the study of 
holomorphic maps branched in three points only by using solely combinatorial
information. Such maps nowadays go by the name of \emph{Belyi maps}.
This does in particular apply to the holomorphic maps $j_{\bar\Gamma}$ we
consider.

Though we are not concerned with the arithmetic implications of Grothendiecks
proposal we still find ample use of its topological and combinatorial content. 

\begin{defin}
The preimage of the line segment $[0,1]$ in $X(1)$ under a map $j_{\bar\Gamma}$ provides a graph embedded in $X(\bar\Gamma)$,
which is the \emph{dessin d'enfant} associated to $\bar\Gamma$.
\end{defin}

To distill the combinatorial information from the dessin let us now get a bit more
technical:

\begin{defin}
A finite \emph{bipartite graph} is a set $E$ of edges and two equivalence
relations
on $E$ with equivalence classes called white
and black vertices, respectively.
\end{defin}

The line segment $[0,1]$ with a black vertex at $0$ and a white vertex at $1$
realizes the unique bipartite graph with only one edge. The preimage $G_{\bar\Gamma}= j_{\bar\Gamma}^{-1}([0,1])$ 
in $X(\bar\Gamma)$ is a bipartite graph with edge set the components of
the preimage of $]0,1[$ and vertices in bijection to the preimages of $0$ and
$1$.

But to take account also of the embedding into $X(\bar\Gamma)$ we have to
keep track also of the counterclockwise cyclic order of edges at each vertex with respect to the orientation of $X(\bar\Gamma)$.

\begin{defin}
A finite \emph{hypermap} is a set $E$ of edges and two permutations
$\sigma,\alpha\in S_E$ generating a transitive subgroup. White (black) vertices of its underlying bipartite graph
are the orbits of the subgroups generate by $\sigma$ ( $\alpha$ ) respectively.
\\
Two hypermaps $E,\sigma,\alpha$ and $E',\sigma',\alpha'$ are \emph{isomorphic}, if there
exists a bijection $\psi:E\to E'$ such that $\psi\sigma=\sigma'\psi$ and 
$\psi\alpha=\alpha'\psi$.
\end{defin} 

To classify modular subgroups Millington \cite{Mil69b} introduced the notion
of \emph{pairing}, which in our terminology is a \emph{rooted} hypermap, i.e.\ a hypermap $E,\sigma,\alpha$ together with a \emph{root} $e_0\in E$ such that 
$\sigma^3=\operatorname{id}_E=\alpha^2$. He result can thus be stated as follows:

\begin{theo}\cite[Theorem 1]{Mil69b}\label{theo_Perm} There is a one-to-one correspondence between subgroups $\bar{\Gamma}$ of finite
index $n$ in $\PSL(2, \ZZ)$ and rooted hypermaps $E, \sigma, \alpha, e_0$ with $\# E=n$, $\sigma^3=\operatorname{id}_E=\alpha^2$
up to isomorphisms preserving the root. \\
$\bar{\Gamma}$ has type $(n; g, h, e_2, e_3)$ if and only if
\begin{enumerate}
\item $\alpha$ fixes $e_2$ letters of $E$,
\item $\sigma$ fixes $e_3$ letters of $E$, and
\item  $\sigma\alpha$ consists of $h$ disjoint cycles.
\end{enumerate}
\end{theo}

Given a hypermap  as in the theorem, there is an induced transitive action of $\PSL(2, \ZZ)$
on $E$ by mapping the torsion generators $ST$ and $S$
to the permutations
$\sigma$ and $\alpha$.
The correspondence then associates the stabilizer group $\bar\Gamma$ of the root
to a rooted diagram, Hence we obtain

\begin{cor}$($cf.\ \cite[Theorem 4.2]{Kul91}\label{classify}  
Conjugacy classes of subgroups $\bar{\Gamma}\subset\PSL(2, \ZZ)$ of
index $n$ are in bijection with hypermaps $E, \sigma, \alpha$ with $\# E=n$, $\sigma^3=\operatorname{id}_E=\alpha^2$
up to isomorphisms.
\end{cor}

Still we are proceeding towards a concrete visualization. Let us consider a hypermap $E, \sigma, \alpha$ and a corresponding conjugacy class represented by a modular subgroup $\bar\Gamma$.

We observe, that the preimage $G_{\bar\Gamma}= j_{\bar\Gamma}^{-1}([0,1])$ coincides with the bipartite graph underlying $E, \sigma, \alpha$. Moreover it is
embedded into the Riemann surface $X(\bar\Gamma)$ with complement a union of $h$ disjoint cells.
Each of these cells is a cyclic branch cover of the complement of $[0,1]$ in $X(1)$ branched only at infinity, with degrees in bijection to the orbit lengths of $\alpha\circ\sigma$.

Conversely, to a bipartite graph on a surface we can associate the hypermap
$E, \sigma, \alpha$, where $E$ is the set of edges, $\sigma$ maps each edge $e$
to the next one in counterclockwise direction at the unique white vertex of $e$ and
$\alpha$ does the same w.r.t.\ black vertices. In particular $\sigma^3=\operatorname{id}_E=\alpha^2$ if and only if  the graph has white (black) vertices
of degrees one and three (one and two) only.

As we consider only the case that $X(\bar\Gamma)$ is of genus $0$, we conclude

\begin{cor}
Conjugacy classes of subgroups $\bar{\Gamma}\subset\PSL(2, \ZZ)$ of
index $n$ and genus $0$ are in bijection with bipartite graphs with $n$ edges and white (black) vertices
of degrees one and three only (one and two only)
embedded into the sphere up to isotopy.
\end{cor}

Using stereographic projection from any point of a face we get planar representative of
any graph on a sphere, from which we can recover the original graph and thus the hypermap $E, \sigma, \alpha$. Note that projections from different cells yield planar
graphs which in general are non isotopic. 

Hence we declare two planar graphs to be equivalent if they are stereographic projections of the same graph on the sphere up to isotopy. In particular each equivalence
class of finite planar graphs contains at most as many planar isotopy classes as it contains faces. To get from one to another, one hast to make a different face to become
the unbounded face.

So finally we classified modular subgroups by equivalence classes of suitable planar
graphs.

\bigskip

For the remainder of the chapter let us focus on torsion free modular subgroups,
hence the case $e_2=0=e_3$. In this case the corresponding planar bipartite graph has no univalent vertices and is called \emph{(white) trivalent} despite the fact that all black vertices are bivalent.

\medskip

There is no torsion free modular subgroup of index $n$, if $n$ is not
divisible by $6$. In case $6 \mid n$ recursive formulae are given by \cite[Theorem 4.2, 4.3]{SB21}
for the number of torsion free modular subgroups of index $n$ and the number of their conjugacy classes.
They are used to compute the initial members of the sequence of such number, which we reproduce
here for index at most $24$:

\begin{table}[!h]
\begin{tabular}{|c||c|c|c|c|}\hline
 $n$&$6$&$12$&$18$&$24$\\
  \hline
 $\sharp$&$4$&$32$&$336$&$4096$ \\
 
$\sharp$ up to conj.&$2$&$6$&$26$&$191$ \\
\hline
\end{tabular}
\caption{} \label{theo_TFMod}\label{Table0}
\end{table}

\begin{rem}\label{rem_numbdes} In the examples \ref{ex_index6}, \ref{ex_index12} and \ref{ex_index18} below we find graphs associated to the conjugacy classes of torsion free modular subgroups  $\bar{\Gamma}$ of index 6, 12 and 18. A  list and the pictures of the 191 graphs associated to the index 24 torsion free modular subgroups can be found in \cite{HHP21}. 
\end{rem}


\begin{exam}\label{ex_index6} Let us consider modular subgroups $\bar{\Gamma}$,  torsion-free, of index $6=6\cdot k $ with $k=1$ and genus $0$. So we have to construct a corresponding planar bipartite graph with $2=6/3$ white vertices. We know that $\bar{\Gamma}$ has $h = 3=k+2$ parabolic generators,  therefore by Theorem \ref{theo_Perm} its graph has $3$ faces, finally, the number of black vertex is $3=6/2$. Moreover, by Theorem \ref{theo_Perm} (3) there is a partition of $n=6$ into $h=3$ positive integers attached to the conjugacy class of $\sigma\alpha$. We have the following possibilities $[4,1,1]$, $[3,2,1]$ or $[2,2,2]$. Here we use the following interpretation: the parts correspond bijectively to faces such that the size of the part is half the number of edges adjacent to the face (an edge is counted twice for the face if that is adjacent on both sides). We check that  $[3,2,1]$ gives not rise to a trivalent graph. Therefore we have only two partitions and indeed only the following graphs up to equivalence.

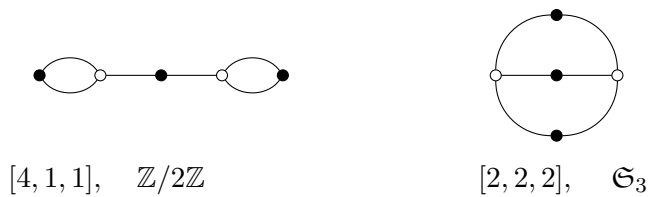
\begin{center}
\begin{tikzpicture}[-latex, node distance = 2cm]
\node[circle,draw, scale=0.4](o){};
\node[circle, draw, fill=black, scale=0.4](o1)[left of= o]{};
\node[circle, draw, fill=black, scale=0.4](a)[right of= o]{};
\node[circle, draw, scale=0.4](b)[right of= a]{};
\node[circle, draw, fill=black, scale=0.4](c)[right of= b]{};
\path[-](a) edge (b);
\path[-](o) edge (a);
\path [-](o1) edge [bend left =55](o);
\path [-](o1) edge [bend right =55](o);
\path [-](b) edge [bend left =55](c);
\path [-](b) edge [bend right =55](c);

\node[circle](o3)[right of= c]{};{};
\node[circle,draw, scale=0.4](o2)[right of= o3]{};{};
\node[circle, draw, fill=black, scale=0.4](a0)[right of= o2]{};
\node[circle, draw, fill=black, scale=0.4](a1)[above of= a0]{};
\node[circle, draw, fill=black, scale=0.4](a2)[below of= a0]{};
\node[circle, draw, scale=0.4](b0)[right of= a0]{};
\path[-](a0) edge (b0);
\path[-](o2) edge (a0);
\path [-](a1) edge [bend right =45](o2);
\path [-](a1) edge [bend left =45](b0);
\path [-](a2) edge [bend left =45](o2);
\path [-](a2) edge [bend right =45](b0);
\end{tikzpicture}
\end{center}
\begin{center}
$[4,1,1], \quad \ZZ/2\ZZ$  \quad \quad \quad \quad \quad \quad \quad \quad \quad        $[2,2,2]$, \quad $\mathfrak{S}_3$
\end{center}

Notice that in the graph on the right the bounded faces are adjacent to two edges,
while the unbounded face has eight adjacent edges, four curved edges and two
straight edges, each counted twice.

Finally notice that we label each graph with the corresponding partition and with its symmetry group as a graph on a sphere if it is non-trivial.
\end{exam}

\begin{exam}\label{ex_index12} Let us consider modular subgroups $\bar{\Gamma}$,  torsion-free, of index $12=6\cdot k $ with $k=2$ and genus $0$. So the graphs now have $4=12/3$ white vertices and $6=12/2$ black vertices. $h=4=k+2$ is the number of parabolic generators,  therefore by Theorem \ref{theo_Perm} the graphs have four faces. By Theorem \ref{theo_Perm} (3) the faces have to represent partition of $12$ into four positive integers. By Beauville \cite{B81} one sees that only the cases $[9,1,1,1],[8,2,1], [6,3,2,1], [5,5,1,1], [4,4,2,2]$ and $[3,3,3,3]$ give rise to a trivalent planar graph. Therefore we have only the following six planar graphs.

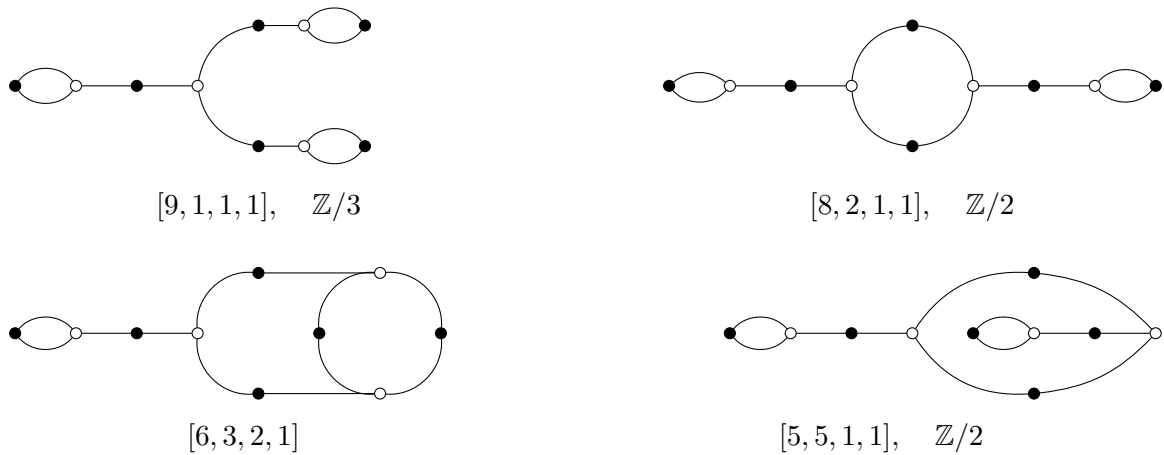
\begin{center}
\begin{tikzpicture}[-latex, node distance = 2cm]

\node[circle](o){};{};
\node[circle](o1)[left of= o]{};
\node[circle,draw, scale=0.4](w1)[above of= o1]{};
\node[circle,draw, scale=0.4](w2)[below of= o1]{};
\node[circle, draw, fill=black, scale=0.4](b1)[right of= w1]{};
\node[circle, draw, fill=black, scale=0.4](b2)[right of= w2]{};
\node[circle, scale=0.3](b3)[left of= o1]{};
\node[circle, draw, fill=black, scale=0.4](b6)[above  of= b3]{};
\node[circle, draw, fill=black, scale=0.4](b7)[below  of= b3]{};
\node[circle,draw, scale=0.4](w3)[left of= b3]{};
\node[circle, draw, fill=black, scale=0.4](b4)[left of= w3]{};
\node[circle,draw, scale=0.4](w4)[left of= b4]{};
\node[circle, draw, fill=black, scale=0.4](b5)[left of= w4]{};
\path [-](w1) edge [bend left =55](b1);
\path [-](w1) edge [bend right =55](b1);
\path [-](w2) edge [bend left =55](b2);
\path [-](w2) edge [bend right =55](b2);
\path [-](w4) edge [bend left =55](b5);
\path [-](w4) edge [bend right =55](b5);
\path[-](w4) edge (b4);
\path[-](b4) edge (w3);
\path[-](w1) edge (b6);
\path[-](w2) edge (b7);
\path [-](w3) edge [bend left =35](b6);
\path [-](w3) edge [bend right =35](b7);
\node[circle](o2)[right of= o]{};
\node[circle, draw, fill=black, scale=0.4](b7)[right  of= o2]{};
\node[circle,draw, scale=0.4](w5)[right of= b7]{};
\node[circle, draw, fill=black, scale=0.4](b8)[right  of= w5]{};
\node[circle,draw, scale=0.4](w6)[right of= b8]{};
\node[circle,  scale=0.4](o3)[right of= w6]{};
\node[circle, draw, fill=black, scale=0.4](b11)[above  of= o3]{};
\node[circle, draw, fill=black, scale=0.4](b12)[below  of= o3]{};
\node[circle,draw, scale=0.4](w7)[right of= o3]{};
\node[circle, draw, fill=black, scale=0.4](b9)[right  of= w7]{};
\node[circle,draw, scale=0.4](w8)[right of= b9]{};
\node[circle, draw, fill=black, scale=0.4](b10)[right  of= w8]{};
\path[-](w5) edge (b8);
\path[-](b8) edge (w6);
\path[-](w7) edge (b9);
\path[-](b9) edge (w8);
\path [-](w8) edge [bend left =50](b10);
\path [-](w8) edge [bend right =50](b10);
\path [-](w5) edge [bend left =50](b7);
\path [-](w5) edge [bend right =40](b7);
\path [-](w6) edge [bend left =40](b11);
\path [-](w6) edge [bend right =40](b12);
\path [-](w7) edge [bend left =40](b12);
\path [-](w7) edge [bend right =40](b11);

\end{tikzpicture}
\end{center}
\begin{center}
$[9,1,1,1], \quad \ZZ/3  \quad \quad \quad\quad\quad \quad \quad \quad \quad \quad \quad\quad \quad \quad \quad$        $[8,2,1,1], \quad \ZZ/2$
\end{center}

\medskip

\begin{center}
\begin{tikzpicture}[-latex, node distance = 2cm]

\node[circle](o){};{};
\node[circle](o1)[left of= o]{};
\node[circle, draw, fill=black, scale=0.4](b1)[left  of= o1]{};
\node[circle, scale=0.4](o2)[left of= b1]{};
\node[circle,draw, scale=0.4](w1)[above of= o2]{};
\node[circle,draw, scale=0.4](w2)[below of= o2]{};
\node[circle, draw, fill=black, scale=0.4](b2)[left  of= o2]{};
\node[circle, scale=0.4](o3)[left of= b2]{};
\node[circle, draw, fill=black, scale=0.4](b3)[above  of= o3]{};
\node[circle, draw, fill=black, scale=0.4](b4)[below  of= o3]{};
\node[circle,draw, scale=0.4](w3)[left of= o3]{};
\node[circle, draw, fill=black, scale=0.4](b5)[left  of= w3]{};
\node[circle,draw, scale=0.4](w4)[left of= b5]{};
\node[circle, draw, fill=black, scale=0.4](b6)[left  of= w4]{};
\path[-](w1) edge (b3);
\path[-](w2) edge (b4);
\path[-](w3) edge (b5);
\path[-](w4) edge (b5);
\path [-](w1) edge [bend left =50](b1);
\path [-](w1) edge [bend right =50](b2);
\path [-](w2) edge [bend right =50](b1);
\path [-](w2) edge [bend left =50](b2);
\path [-](w3) edge [bend left =50](b3);
\path [-](w3) edge [bend right =50](b4);
\path [-](w4) edge [bend left =50](b6);
\path [-](w4) edge [bend right =50](b6);
\node[circle, scale=0.1](o4)[right of= o]{};
\node[circle, draw, fill=black, scale=0.4](bb1)[right  of= o4]{};
\node[circle,draw, scale=0.4](ww1)[right of= bb1]{};
\node[circle, draw, fill=black, scale=0.4](bb2)[right  of= ww1]{};
\node[circle,draw, scale=0.4](ww2)[right of= bb2]{};
\path [-](ww1) edge [bend left =50](bb1);
\path [-](ww1) edge [bend right =50](bb1);
\path[-](ww1) edge (bb2);
\path[-](ww2) edge (bb2);
\node[circle, draw, fill=black, scale=0.4](o5)[right of= ww2]{};
\node[circle,draw, scale=0.4](ww3)[right of= o5]{};
\node[circle, draw, fill=black, scale=0.4](bb3)[above  of= ww3]{};
\node[circle, draw, fill=black, scale=0.4](bb4)[below  of=ww3]{};
\node[circle, draw, fill=black, scale=0.4](bb5)[right  of= ww3]{};
\node[circle,draw, scale=0.4](ww4)[right of= bb5]{};
\path [-](ww3) edge [bend left =50](o5);
\path [-](ww3) edge [bend right =50](o5);
\path [-](ww2) edge [bend left =30](bb3);
\path [-](ww2) edge [bend right =30](bb4);
\path[-](ww3) edge (bb5);
\path [-](ww4) edge [bend right =20](bb3);
\path [-](ww4) edge [bend left =20](bb4);
\path[-](ww4) edge (bb5);
\end{tikzpicture}
\end{center}
\begin{center}
$[6,3,2,1] \quad \quad \quad \quad \quad \quad\quad\quad \quad \quad \quad \quad \quad \quad\quad \quad$        $[5,5,1,1], \quad \ZZ/2$
\end{center}

\medskip 

\begin{center}
\begin{tikzpicture}[-latex, node distance = 2cm]
\node[circle](o){};{};
\node[circle](o1)[left of= o]{};
\node[circle, draw, fill=black, scale=0.4](b1)[left  of= o1]{};
\node[circle,  scale=0.4](o3)[left of= b1]{};
\node[circle,draw, scale=0.4](w1)[above of= o3]{};
\node[circle,draw, scale=0.4](w2)[below of= o3]{};
\node[circle, draw, fill=black, scale=0.4](b2)[left  of= o3]{};
\node[circle,  scale=0.4](o4)[left of= b2]{};
\node[circle, draw, fill=black, scale=0.4](b5)[above  of= o4]{};
\node[circle, draw, fill=black, scale=0.4](b6)[below  of= o4]{};
\node[circle, draw, fill=black, scale=0.4](b3)[left  of= o4]{};
\node[circle,  scale=0.4](o5)[left of= b3]{};
\node[circle,draw, scale=0.4](w3)[above of= o5]{};
\node[circle,draw, scale=0.4](w4)[below of= o5]{};
\node[circle, draw, fill=black, scale=0.4](b4)[left  of= o5]{};
\path[-](w1) edge (b5);
\path[-](w3) edge (b5);
\path[-](w2) edge (b6);
\path[-](w4) edge (b6);
\path [-](w1) edge [bend left =50](b1);
\path [-](w1) edge [bend right =50](b2);
\path [-](w2) edge [bend left =50](b2);
\path [-](w2) edge [bend right =50](b1);
\path [-](w3) edge [bend left =50](b3);
\path [-](w3) edge [bend right =50](b4);
\path [-](w4) edge [bend left =50](b4);
\path [-](w4) edge [bend right =50](b3);
\node[circle, scale=0.1](o2)[right of= o]{};
\node[circle,  scale=0.4](v2)[right of= o2]{};
\node[circle, draw, scale=0.4](ww1)[above  of= o2]{};
\node[circle, draw, scale=0.4](ww2)[right  of= v2]{};
\node[circle,  fill=black, scale=0.4](v4)[below of= ww2]{};
\node[circle,   scale=0.4](v3)[right of= ww2]{};
\node[circle,  fill=black, scale=0.4](v5)[left of= v4]{};
\node[circle,  fill=black, scale=0.4](v7)[right of= v4]{};
\node[circle,  scale=0.20](h)[above of= v7]{};
\node[circle,   fill=black, scale=0.4](v8)[above of= h]{};
\node[circle,  scale=0.20](h1)[above of= v5]{};
\node[circle,   fill=black, scale=0.4](v9)[above of= h1]{};
\node[circle, draw, scale=0.4](ww4)[below of= v4]{};
\node[circle, scale=0.3](o5)[right of= v3]{};
\node[circle, draw, scale=0.4](ww5)[above of= o5]{};
\node[circle, scale=0.1](o6)[above of= ww2]{};
\node[circle,  fill=black, scale=0.4](v6)[above of= o6]{};
\path[-](ww2) edge (v4);
\path[-](ww2) edge (v8);
\path[-](ww5) edge (v8);
\path[-](ww1) edge (v9);
\path[-](ww2) edge (v9);
\path[-](ww4) edge (v4);
\path[-](ww5) edge[bend right =05] (v6);
\path [-](ww4) edge [bend left =10](v5);
\path [-](ww1) edge [bend left =05](v6);
\path [-](ww1) edge [bend right =10](v5);
\path [-](ww4) edge [bend right =10](v7);
\path [-](ww5) edge [bend left =10](v7);
\path [-](w4) edge [bend right =50](b3);
\end{tikzpicture}
\end{center}
\begin{center}
$[4,4,2,2], \quad \ZZ/2 \times \ZZ/2   \quad \quad \quad \quad \quad \quad \quad \quad\quad \quad \quad \quad$        $[3,3,3,3], \quad A_4$
\end{center}

\end{exam}

\begin{exam} \label{ex_index18} Let us consider modular subgroups $\bar{\Gamma}$,  torsion-free, of index $18=6\cdot k $ with $k=3$ and genus $0$. So the graphs now have $6=18/3$ white vertices and $9=18/2$ black vertices, $h=5=k+2$ is the number of  faces. By Theorem \ref{theo_Perm} (3) the faces have to represent partition of $18$ into five positive integers. In \cite{} planar graphs representing the $26$ equivalence classes are given. We reproduce the graphs except in cases where we chose a different unbounded region to highlight the symmetries. We do not draw the black vertices here, it is tacitly understood that on each edge between white vertices there is always a black vertex. 

\begin{center}

\end{center}
\begin{center}
$[5,5,3,3,2] , \quad \ZZ/2 \quad   \quad \quad \quad \quad\quad \quad \quad \quad$        $[4,4,4,3,3] , \quad S_3$
\end{center}
\end{exam}

Notice that two graphs have the same associated partition $[7,7,2,1,1]$ but represent
distinct equivalence classes of graphs.



\section{Expressions for the Euler number}\label{sec_Euler}

The results collected so far are sufficient to derive a first characterization of the modular monodromy groups $\bar{\Gamma}$
which can occur for the elliptic surfaces $\EE$ we consider. The strategy we want to pursue to classify the monodromy groups $\Gamma \subset \SL(2,\ZZ)$ is the following: First describe all possible groups $\bar{\Gamma} \subset \PSL(2,\ZZ)$. In order to classify the former we consider the corresponding modular curves $X(\bar{\Gamma})$ and the covering of $\PP^1$ with the properties described in Proposition \ref{prop_FactJ}. 

Following \cite{BT03}, we call $S_{\Gamma}$ a \emph{$j$-modular elliptic surface} associate to $\bar{\Gamma} \subset \PSL(2,\ZZ)$  over the modular curve $X(\bar{\Gamma})$ if $j(S_{\Gamma})=j_{\bar{\Gamma}}$.
We have the following diagram.
\[
\xymatrix{
\EE  \ar[d]  & S_{\Gamma} \ar[d] & 
\\ 
\PP^1 \ar[r]^{l:1} \ar@/^-2.0pc/@[black][rr]_{j_{\EE}} & X(\bar{\Gamma}) \ar[r]^{j_{\bar{\Gamma}}}& \PP^1.
}
\]

This last diagram induces a natural map $\pi_1\big(\PP^1\setminus j(\EE)^{-1}\{0,1,\infty\} \big) \to \pi_1\big(X(\bar{\Gamma})\setminus j_{\bar{\Gamma}}^{-1}\{0,1,\infty\}\big)$ and  a commutative diagram of monodromy homomorphisms
\[
\xymatrix{
\pi_1\big(\PP^1\setminus B \big)  \ar[d]   \ar[r]^{} & \Gamma \ar[d]
\\ 
\pi_1\big(X(\bar{\Gamma})\setminus j_{\bar{\Gamma}}^{-1}\{0,1,\infty\}\big) \ar[r]^{} & \bar{\Gamma} .
}
\]

We can compare the lifting of the elliptic fibration $S_{\Gamma}$ to $\PP^1$ with $\EE$. As in \cite{BT03} there are two cases:

\begin{enumerate}
\item The $1:1$ case. In this case $\Gamma \cong \bar{\Gamma}$ and $\EE$ is (fiberwise birationally) isomorphic to one of the $j$-modular surfaces corresponding to a section $\bar{\Gamma} \to \SL(2,\ZZ)$.
\item The general case. In this case $\EE$ is (fiberwise birationally)  obtained as a composition of a pullback of a corresponding $S_{\Gamma}$ to $\PP^1$ followed by an even number of twists \cite[Section 5.9.1]{SS}. 

\end{enumerate}

Let $j_{\bar{\Gamma}}\colon X(\bar{\Gamma}) \to \PP^1$ be a holomorphic map. Moreover, let $B_{\bar{\Gamma}}=j_{\bar{\Gamma}}^{-1}(\{0,1,\infty\})$ and then $\rho'\colon \pi_1(\PP^1 \setminus B_{\bar{\Gamma}}) \to \PSL(2,\ZZ)$ be the associated monodromy. Then using the notation of the Section \ref{sec_jmod} we have by \cite[Proposition 2.3]{Nori} there always exists a lift $\rho$ of $\rho'$ to $\SL(2,\ZZ)$. Notice that it is not unique. Indeed, we will have examples where we encounter different lifts.

However, if we choose free generators of $ \pi_1(\PP^1 \setminus B_{\bar{\Gamma}})$ corresponding to loops around all but one of the elements in $B_{\bar{\Gamma}}$, there is a preferred lift $\rho$. Indeed, $\rho'$ assigns  a matrix class to each of the loops of the chosen generators and $\rho'$ is obtained  assigning the matrix corresponding to smallest Euler number in Table \ref{Table01}. 

Note that the matrix assigned to the loop around the remaining element in $B_{\bar{\Gamma}}$ is then determined and it may be the matrix with smaller or with bigger Euler number.  Since matrices with bigger Euler number correspond to a $*$-fibre we get the following lemma. 

\begin{lem}\label{lem_homInvStarFib} Given a functional invariant there exists a compatible homological invariant with either no $*$-fibers or 1 $*$-fibre. 
\end{lem}

We need a precise formula for the Euler number of $\EE$ for the calculation we are going to do.  Denote as before $l=\deg j_{\EE}$ and consider the 
following diagram
\[
\xymatrix{
\EE  \ar[d]^f  & & 
\\ 
\PP^1 \ar[r]^{l:1} \ar@/^-2.0pc/@[black][rr]_{j(\EE)} &  X(\bar{\Gamma}) \cong \PP^1 \ar[r]^{j_{\bar{\Gamma}}}& X(1) \cong \PP^1. 
}
\]

By Kodaira theory on elliptic fibration  we can calculate the Euler number of $\EE$ only knowing the Euler number of the  singular fibers of the elliptic fibration. Let us denote by $B \subset \PP^1$ the set of points $b$ such that $f^{-1} (b)=F_b$ is a singular fibre then we have
\[
e(\EE) = \sum_{b \in B} e(F_b).
\]
Using the definition of $j(\EE)$ we can decompose $B$ as the disjoint union of $B^\infty=j(\EE)^{-1}(\infty)$ and its complement $B^{tor}$, thus
\[
e(\EE) = \sum_{b \in B^\infty}e(F_b)+ \sum_{b \in B^{tor}}e(F_b).
\]
The possible singular fibers can be $*$-fibers or not. Using Table \ref{table:kodaira} we see that if $F^*_b$ is a $*$-fibre then $e(F^*_b)=6+e(F_b)$ where $F_b$ is the same fibre without the $*$ -- if different from $II^*$ and $IV^*$ which are exchanged when we omit the $*$ -- . Therefore we can collect all the $*$-fibers together. This yields
\[
e(\EE) = 6\cdot \sharp\left\{ *\textrm{-fibrers} \right\}+ \sum_{b \in B^\infty}e(F_b)+ \sum_{b \in B^{tor}}e(F_b)
\]
where now in the summands there are no more $*$-fibers. Using the decomposition of $j(\EE)= j_{\bar{\Gamma}}\circ j_{\EE}$ and recalling that  $\deg(j_{\bar{\Gamma}})=[\PSL(2,\ZZ): \bar{\Gamma}],\deg( j_{\EE})=l$, we deduce that
\begin{equation}\label{eq_EulerMore}
e(\EE)=6\cdot \sharp\left\{ *\textrm{-fibrers} \right\}+l \cdot \textrm{index }(\bar{\Gamma})+\sum_{b \in B^{tor} }e(F_b).
\end{equation}
To understand the last summand, let us analyze the set  $B^{tor}$. Every element $b \in B^{tor}$  is mapped by $j(\EE)$ to $0$ or $1$. Suppose first that $b \mapsto 1$ then by the factorization of $j(\EE)$
\[
b \mapsto j_{\EE}(b)=:p \mapsto 1.
\]
There are two alternatives for $p$: it is either a $2$-torsion point or it is not. Suppose first $p$ is a $2$-torsion point, then the multiplicity of $j(\EE)$ at $b$ is one more than the ramification $r$ of $j_{\EE}$ at $b$. So locally we have the following $j(\EE)$-expansion, which determines the Euler number of the fibre according to Table \ref{table:kodaira}
\[
b \mapsto 1+b^{r+1} \Rightarrow \begin{cases} e(F_b)=3 & \textrm{if } r+1\equiv 1 \mod 2; \\
 e(F_b)= 0 & \textrm{if } r+1\equiv 0 \mod 2. \\

\end{cases}
\]
The Euler number contribution at the point $b$ can thus be expressed by  
\[
6 \left\{\frac{r+1}{2} \right\} := 6 \left( \frac{r+1}{2} - \left[ \frac{r+1}{2} \right] \right),
\]
the brackets $\{\}$ extracting the fractional part of a rational number.
 
If $p$ is not $2$-torsion then $j_{\bar{\Gamma}}$ is simply ramified at $p$ and thus of even multiplicity. Hence the multiplicity of $j(\EE)$ is also even. This yields $e(F_b)=0$.

Secondly,  suppose  that $b \mapsto 0$ then by factorization of $j(\EE)$
\[
b \mapsto j_{\EE}(b)=p \mapsto 0.
\]
There are two alternatives for $p$: it is either a $3$-torsion point or it is not. Suppose first $p$ is a $3$-torsion point and the ramification is $t$ at $b$, then the multiplicity of $j_{\EE}$ at $b$ is $t+1$.  Again we get a local expansion of 
$j(\EE)$ and the Euler number of the fibre according to Table \ref{table:kodaira}
\[
b \mapsto b^{t+1} \Rightarrow \begin{cases} e(F_b)=4 & \textrm{if } t+1\equiv 2 \mod 3; \\
e(F_b)= 2 & \textrm{if } t+1\equiv 1 \mod 3; \\
e(F_b)= 0 & \textrm{if } t+1\equiv 0 \mod 3; \\
\end{cases}
\]
Hence the Euler number contribution is $6\cdot \left\{\frac{t+1}{3} \right\}$.

If $p$ is not $3$-torsion then as before we get a trivial contribution $e(F_b)=0$.

We have therefore proven the following proposition.
\begin{prop}\label{prop_EulerEq}

For a Jacobian fibration $\EE \to \PP^1$, we have
\begin{equation}\label{eq_Euler1}
e(\EE)=6\cdot \sharp\left\{ *\textrm{-fibres} \right\}+l \cdot \textrm{index }(\bar{\Gamma})+ 6 \left(\sum_{i\in T_2}\left\{\frac{r_i+1}{2} \right\}+\sum_{j\in T_3}\left\{\frac{t_j+1}{3} \right\} \right).
\end{equation}
where $T_2$ is the set of points mapping to $2$-torsion points and $r_i, i\in T_2$ the
corresponding ramification indices, and $T_3$, $t_j, j\in T_3$ similarly for $3$-torsion points.
\end{prop}

\begin{cor}\label{cor_Euler}
If $S_{\Gamma}$ is a $j$-modular elliptic surface associated to $\bar{\Gamma} \subset \PSL(2,\ZZ)$ 
then
\begin{equation}
e(S_{\Gamma})=6\cdot \sharp\left\{ *\textrm{-fibres} \right\}+ \textrm{index }(\bar{\Gamma})+ 6 \left(\frac{e_2}2 + \frac{e_3}3 \right).
\end{equation}
\end{cor}

\begin{rem}\label{rem_Index24} Since we are interested in elliptic fibration over K3 surfaces, whose Euler number is $24$ we deduce the following for torsion free $\bar{\Gamma}$:
\begin{enumerate}
\item $\textrm{index }(\bar{\Gamma}) \leq 24$.  Equality occurs when $l=1$.
\item $6$ divides $\textrm{index }(\bar{\Gamma})$. 
\end{enumerate}
Therefore for torsion free $\bar{\Gamma}$ we have $\textrm{index}(\bar{\Gamma}) \in \{6,12,18,24\}$.
\end{rem}

\section{Dessins of subgroups containing torsion}\label{sec_CTG}

To pass from a dessin corresponding to a subgroup possibly containing torsion to a dessin of a torsion free subgroup we perform the following substitutions:
\begin{enumerate}
\item
A univalent white vertex is replaced by a trivalent white vertex such that the new edges meet at
an additional black vertex as shown on the right.
\begin{center}
\begin{tikzpicture}[-latex, node distance = 2cm]
\node[circle,draw, scale=0.4](o){};
\node(o1)[left of= o]{};
\node(a)[right of= o]{};
\node[circle, draw, scale=0.4](b)[right of= a]{};
\node[circle, draw, fill=black, scale=0.4](c)[right of= b]{};
\path[-](a) edge (b);
\path[-](o1) edge (o);
\path [-](b) edge [bend left =55](c);
\path [-](b) edge [bend right =55](c);
\end{tikzpicture}
\end{center}

\item
A univalent black vertex is replaced by a graph consisting of three edges, a white and
two black vertices as shown on the right.
\begin{center}
\begin{tikzpicture}[-latex, node distance = 2cm]
\node[circle,draw,fill=black, scale=0.4](o){};
\node(o1)[left of= o]{};
\node(o2)[right of= o]{};
\node[circle,draw, fill=black, scale=0.4](a)[right of= o2]{};
\node[circle, draw, scale=0.4](b)[right of= a]{};
\node[circle, draw, fill=black, scale=0.4](c)[right of= b]{};
\path[-](a) edge (b);
\path[-](o1) edge (o);
\path[-](o2) edge (a);
\path [-](b) edge [bend left =55](c);
\path [-](b) edge [bend right =55](c);
\end{tikzpicture}
\end{center}
\end{enumerate}

Let $\bar{\Gamma}$ be a modular subgroup of finite index and consider the rooted dessin associated to it. Perform all possible substitutions defined above to obtain a rooted dessin associated to a class of torsion free modular subgroups.
The corresponding mapping
\[
\tf\colon \bar{\Gamma} \mapsto \bar{\Gamma}^{\tf}
\]
retracts the set of rooted dessins of modular subgroups onto the set of rooted dessins of torsion free modular subgroups. It induces a retraction from the set of conjugacy
classes of modular subgroups onto the set of conjugacy classes of torsion free modular subgroups.

\begin{defin} The \emph{minimal Euler number} of a finite index modular subgroup $\bar{\Gamma} \subset \PSL(2, \ZZ)$ is the integer
\[
e(\bar{\Gamma}):=\min\{ e(\EE)| \EE \textrm{ has modular monodromy } \bar{\Gamma} \}.
\] 
\end{defin}

\newcommand{\indp}{\operatorname{ind}_\PSL}

With this new terminology the following lemma just follows from Corollary \ref{cor_Euler}:

\begin{lem}  \label{euler_index} 
Let $\bar{\Gamma}$ be a torsionfree modular subgroup of finite index and genus $0$, then
\[
e(\bar{\Gamma})=\begin{cases}
\indp(\bar{\Gamma}) & \textrm{ if } 12| \indp(\bar{\Gamma}) \\
\indp(\bar{\Gamma})+6 & \textrm{ else}.
\end{cases}
\]
\end{lem}

\begin{rem}
\begin{equation}\label{eq_index}
\begin{split}
 \textrm{index}(\bar{\Gamma})+3e_2+2e_3 
= & \textrm{index}(\bar{\Gamma}^{\tf})
\end{split} 
\end{equation}
since the index of a modular subgroup is equal to the number of edges in the corresponding graph, and we add two edges for each of the $e_3$ univalent white vertices and $3$ edges for the univalent black vertices to get the graph of $\bar\Gamma^{\tf}$
\end{rem}

\begin{prop}  \label{prop_euler1} 
Let $\bar{\Gamma}$ be a modular subgroup of finite index, and $\bar{\Gamma}^{\tf}$ the torsion free group associated to $\bar{\Gamma}$,  then
\[
e(\bar{\Gamma}) = e(\bar{\Gamma}^{\tf}),
\]
i.e., an elliptic fibration of Euler number $e(\bar{\Gamma})$ with modular monodromy $\bar{\Gamma}$ exists if and only if an elliptic fibration of the same Euler number with modular monodromy $\bar{\Gamma}^{\tf}$ exists.
\end{prop}

\begin{proof}

To prove the inequality $e(\bar{\Gamma})\leq e(\bar{\Gamma}^{\tf})$ it
suffices to find one surface with modular mono\-dromy
$\bar\Gamma$ such that the numbers are equal: \\
Let $S_{\bar{\Gamma}}$ be an elliptic fibration such that $j_{\EE} = j_{\bar{\Gamma}}$
is its functional invariant. We can calculate $e(S_{\bar{\Gamma}})$ -- which is a number divisible by $12$ -- using \eqref{eq_Euler1} with $l=1$ and the equation \eqref{eq_index}.
\begin{equation}\label{eq_minH}
\begin{split}
e(S_{\bar{\Gamma}}) = & 6\cdot \sharp\left\{ *-\textrm{fibres} \right\}+\textrm{index}(\bar{\Gamma})+ 6 (\frac{1}{2}e_2+\frac{1}{3}e_3) \\
= & 6\cdot \sharp\left\{ *-\textrm{fibres} \right\}+\textrm{index}(\bar{\Gamma})+3e_2+2e_3 \\
= & 6\cdot \sharp\left\{ *-\textrm{fibres} \right\}+\textrm{index}(\bar{\Gamma}^{\tf})
\end{split} 
\end{equation}
The number of $*$-fibers is determined by the compatible homological invariant of $S_{\bar{\Gamma}}$.  By Lemma \ref{lem_homInvStarFib} it can be chosen such that this number is either $0$ or $1$ in such a way, that the euler number is divisible by 12,
thus providing an elliptic surface with euler number $e(\bar{\Gamma}^{\tf})$.
\medskip

To prove the inequality $e(\bar{\Gamma})\geq e(\bar{\Gamma}^{\tf})$ it must
be shown, that for all surfaces
the number on the right is a lower bound:
From the diagram above we get
\[
\xymatrix{
\EE  \ar[d]  & & 
\\ 
\PP^1 \ar[r]^{l:1} 
& \PP^1 \ar[r]^{j_{\bar{\Gamma}}}& \PP^1 
}
\]
and  \eqref{eq_Euler1} implies:
\begin{equation}\label{eq.mainInEuler}
e(\EE) \geq \deg(j_{\bar{\Gamma}}) l=(\deg j_{\bar{\Gamma}^{\tf}}-2e_3-3e_2)l.
\end{equation}
Observe that the positive integer $k$ with $\indp(\bar\Gamma^{\tf})=6k$ provides a bound
\begin{equation}\label{eq.EulerProp} 
e_2+e_3 
\leq k+1.
\end{equation}
Indeed, $e_2+e_3$ is the number of univalent vertices, each univalent vertex of the graph of $\bar\Gamma$ contributes one loop in the graph of $\bar\Gamma^{\tf}$ and 
the number of faces for the dessin
associated to $\bar\Gamma^{\tf}$ is equal to $h=k+2$ for genus $0$. Since a loop is only
adjacent to two edges there must be at least one face which is not a loop,
hence the number of loops is at most $k+1$.
\medskip

In a case by case argument we next establish the strict inequality
\begin{equation}\label{eq.strict} 
e(\EE) \quad > \quad 6k - 6
.
\end{equation}
\paragraph{\bf Case $l=1$}
From \eqref{eq_minH} we get immediately $e(S_{\bar\Gamma})\geq 
\textrm{index}(\bar\Gamma^{\tf})=6k > 6k -6$.
\paragraph{\bf Case $l\geq2$, $e_2\leq k$}
By \eqref{eq.EulerProp} we get $e_3\leq k+1-e_2$, so \eqref{eq_minH} implies
\[
\begin{array}{cclcl}
e(\EE) & \geq & 2(6k -3e_2-2e_3)\\
 & \geq & 2(6k-3e_2 -2k-2+2e_2) & = & 2(4k -e_2-2) \\
 & \geq & 2(4k -k-2) & = & 6k -4 \\
 & > & 6k-6
\end{array}
\]
\paragraph{\bf Case $l=2$, $e_2= k+1>2$}
By \eqref{eq.EulerProp} we get $e_3=0$, equation \eqref{eq_Euler1} then implies
\[
e(\EE)\geq 2(6k-3(k+1))+ 6\cdot \sum_{i\in T_2}\left\{\frac{r_i+1}{2} \right\}
\]

The last sum is over the preimages of the two torsion points. Since their number is equal to $e_2 \geq 3$ there are at least three summands. By Riemann-Hurwitz the map
$\PP^1\stackrel{l:1}{\to} \PP^1 $ is ramified at two points, thus for at most two points in $T_2$ the $r_i$ may be strictly bigger than $0$.  Hence the sum gives a strictly positive contribution. Therefore,

\paragraph{\bf Case $l>2$, $e_2= k+1>2$}
Again $e_3=0$, thus equation \eqref{eq_minH} implies
\[
e(\EE)\geq 3(6k-3(k+1)) = 9k -9 = 6k + 3(k-3) \geq 6k -3 > 6k -6
\]
In \eqref{eq_index} the values $k=1, e_2=2$ imply $0$ for the index of $\bar\Gamma$.
Such a group can not exists, so the above case cover everything that is possible.
\medskip

Having established $e(\EE)>6k-6$ in all cases, we compare with
\[
e(\bar{\Gamma}^{\tf})=\begin{cases} 6k & \textrm{ if } k\equiv 0 \mod 2, \\
6k+6 & \textrm{ if } k\equiv 1 \mod 2. \\
\end{cases}
\]
Since $e(\EE)$ is divisible by $12$ we have: if $k\equiv 1 \mod 2$ then $e(\EE) > 6k-6$
implies $e(\EE)\geq 6k+6$.
If  $k\equiv 0 \mod 2$ then $e(\EE)>6k-6$ implies $e(\EE) \geq 6k$. 
In both cases $e(\EE) \geq e(\bar{\Gamma}^{\tf})$ as claimed.
\end{proof}

\begin{theo} \label{theo_main42} 
A modular subgroup $\bar{\Gamma}$ is the modular monodromy of a regular elliptic surface $\EE$ of Euler
number $e(\EE)=n$ if and only if $n$ is a multiple of 12 and
\begin{enumerate}
\item
$\bar\Gamma$ is of genus $0$
\item
$\bar\Gamma^{\tf}$ is of index at most $n$.
\end{enumerate}
which is equivalent to the condition on the corresponding dessin 
to be planar with at most $e(\EE)$ edges.
\end{theo}

\begin{proof}
Given a regular elliptic surface $\EE$ with modular monodromy $\bar\Gamma$, then $e(\EE)$ is divisible
by 12, Morover $\bar\Gamma$ is of genus $0$, since the regularity implies the rationality of the base.
Finally by definition, by Prop.\ref{prop_euler1} and by Lemma \ref{euler_index} respectively
\[
e(\EE) \quad \geq \quad 
e(\bar\Gamma) \quad = \quad
e(\bar\Gamma^{\tf}) \quad \geq \quad
\indp(\bar\Gamma^{\tf})
\]
Conversely given a modular subgroup $\bar\Gamma$ then $e(\bar\Gamma)\leq \indp(\bar\Gamma^{\tf})+6$
by the Prop.\ and the Lemma again. So by definition there exists a regular elliptic surface $\EE'$ with
modular monodromy $\bar\Gamma$ and euler number the smallest integer $n'$ divisible by 12, which
is at least equal to $\indp(\bar\Gamma^{\tf})$. Hence given $n\geq n'$ divisible by 12, there is $\EE$,
equal to $\EE'$ or obtained by replacing $(n-n')/6$ fibres by starred fibres.

\end{proof}

Since elliptic K3 surfaces are precisely the regular Jacobian elliptic surfaces with euler number 24 we
get the following corollary.

\begin{cor}
A modular subgroup $\bar{\Gamma}$ is the modular monodromy of an elliptic K3 surface $\EE$ if and only if
\begin{enumerate}
\item
$\bar\Gamma$ is of genus $0$,
\item
$\bar\Gamma^{\tf}$ is of index at most $24$.
\end{enumerate}
which is equivalent to ask the corresponding dessin 
to be planar with at most $24$ edges.
\end{cor}

Thanks to the theorem we are able to classify all torsion modular subgroup that occur as modular group of a Jacobian fibration with index at most 24. Indeed using Theorem \ref{theo_TFMod} we simply have to calculate associated torsion subgroups using the association of the theorem. 


\begin{prop} \label{prop_jfactor1824} 
There is no K3 elliptic fibration with $j$-factor $l$ strictly larger than 1,
if the modular monodromy $\bar\Gamma$ has $\mathrm{index}(\bar\Gamma^{\tf})$ equal
$18$ or $24$ and $e_2=0$.
\end{prop}

\begin{proof} We exploit  equation \eqref{eq.mainInEuler} with j-factor $l$  and 
$\mathrm{index}(\bar\Gamma^{\tf})=6k$ with $k=3$ or $k=4$ so
\[
e(\EE) \geq (6k-2e_3)\cdot l
\]
First, suppose $l\geq3$. To have $e(\EE)=24$ we need $2e_3\geq16$ in case $k=4$,
resp.\ $2e_2\geq 10$ in case $k=3$. Both can not occur since they contradict \eqref{eq.EulerProp}.
Similarly $l=2$ with $k=4$ can not occur, as we need $2e_3\geq12$ then.

Since we need $2e_3\geq6$ in case $l=2,k=3$ we are left to use the
more refined equation 
\eqref{eq_Euler1} that implies
\[
24\geq 2(18-2e_3)+ 6\cdot \sum_{i\in T_3}\left\{\frac{t_i+1}{3} \right\}
\]

The last sum is over the preimages of the $3$-torsion points. Since their number is equal to $e_3 \geq 3$ there are at least $e_3$ summands. 
With $l=2$ a strict upper bound on the ramification indices $t_i$ all summand contribute
and we get
\[
24 \geq 36 - 4 e_3 + 2 e_3 \quad \implies \quad 2e_3\geq 12
\]
which again contradicts \eqref{eq.EulerProp}.
\end{proof}

\section{Counting groups}\label{sec_CG}

We are now in the position to classify the monodromy groups of Jacobian fibration. As already remarked in this work, we will focus on calculating the number of conjugate classes of these groups. This calculation boils down to non-trivial combinatorial considerations, as we shall see. 

Let us consider the following sets of conjugacy classes (here $\sim$ is the conjugacy relation):
\[
\widetilde{K}:=\{ \Gamma \subset \SL(2,\ZZ) \mid \Gamma \textrm{ is the monodromy of a K3}\}/_{\sim_{\SL(2,\ZZ)}}.
\]
\[
K:=\{ \bar{\Gamma} \subset \PSL(2,\ZZ) \mid j_{\bar{\Gamma}} \textrm{ is the modular part of a K3 $j$-invariant}\}/_{\sim_{\PSL(2,\ZZ)}}.
\]
In addition, we consider the following subset of $K$
\[{
K^{\tf}:=\{ \bar{\Gamma} \subset \PSL(2,\ZZ) \mid \bar{\Gamma} \in K \textrm{ and torsion free}\}/_{\sim_{\PSL(2,\ZZ)}}.}
\]
Note that elements in $K$ correspond bijectively to certain planar dessins.

To determine the cardinality of these sets we want to use the induced maps
\[
\bm{\xi} : \widetilde{K} \to K,\qquad
\tf: K \to K^{\tf}.
\]
The idea is then simply to start with a count of elements in $K^{\tf}$ and bootstrap
to $K$ and $\widetilde{K}$ by adding the finite fibre cardinalities for both maps in
some coherent way. 
By Table \ref{theo_TFMod} we know the cardinality of each set  $K^{\tf}_{i}$ for 
$i\in \{6,12,18,24\}$ in the decomposition of $K^{\tf}$ according to index:
\[
K^{\tf}= K^{\tf}_6 \cup K^{\tf}_{12} \cup K^{\tf}_{18} \cup K^{\tf}_{24}.
\] 
Each subset has the cardinality of the corresponding set of dessins, 
see Remark \ref{rem_numbdes}
and proceed separately with the counting arguments.

In each case we first calculate the number of subgroup classes in the fibres of $\tf$ and secondly those in the fibres of $\bm{\xi}$. The former numbers, index by index, can be deduced from the number of possible substitution on torsion free diagrams using the rules established in the previous section. To apply the rules we stratify the torsion free sets according to the number of loops they contain and, more delicate issue, take into account all possible symmetries.

\subsection{Index 6 case} This case is done by hands. 

\begin{prop}\label{prop_ix6Tors} The cardinality of $K_6$ is $6$, where $K_6$ is the
preimage of $K_6^\tf$. To be precise there are $2$ torsion free subgroups and $4$ with torsion in $K_6$.
\end{prop} 

\begin{proof}
The dessins in Example \ref{ex_index6} give the two elements of $K_6^\tf$. Since
in case $[2,2,2]$ there are no loops, subgroup with torsion only occur in case $[4,1,1]$.
Taking into account the $\ZZ/2\ZZ$ symmetry we get the following additional dessins:
\begin{center}
\begin{tikzpicture}[-latex, node distance = 2cm]
\node[circle,draw, scale=0.4](o){};
\node[circle, draw, fill=black, scale=0.4](o1)[left of= o]{};
\node[circle, draw, fill=black, scale=0.4](a)[right of= o]{};
\node[circle, draw, scale=0.4](b)[right of= a]{};
\path[-](a) edge (b);
\path[-](o) edge (a);
\path [-](o1) edge [bend left =55](o);
\path [-](o1) edge [bend right =55](o);
\node[circle,draw, scale=0.4](p)[right of= a]{};
\node[circle, draw, fill=black, scale=0.4](v1)[right of= p]{};
\node[circle, draw, scale=0.4](v2)[right of= v1]{};
\node[circle, draw, fill=black, scale=0.4](v3)[right of= v2]{};
\path [-](v1) edge [bend left =55](v2);
\path [-](v1) edge [bend right =55](v2);
\path[-](v2) edge (v3);

\node[circle, draw, scale=0.4](w1)[right of= v3]{};
\node[circle, draw, fill=black,  scale=0.4](w2)[right of= w1]{};
\node[circle, draw, scale=0.4](w3)[right of= w2]{};
\path[-](w1) edge (w2);
\path[-](w2) edge (w3);

\node[circle, draw, scale=0.4](q1)[right of= w3]{};
\node[circle, draw, fill=black, scale=0.4](q2)[right of= q1]{};
\path[-](q1) edge (q2);
\end{tikzpicture}
\end{center}
This completes the proof.
\end{proof}

\begin{prop} The cardinality of the subset $\widetilde{K}_6$ of classes in $\widetilde{K}$ mapping to $K_6^\tf$ is 14.
\end{prop}

\begin{proof} 
Let us consider the torsion free groups  $\bar{\Gamma} \in K_6$ first. Then we have two graphs $[4,1,1]$  and $[2,2,2]$. In both cases we have three points in $B_{\bar{\Gamma}}$. 
The groups $\bar{\Gamma}$
have the following presentation $\langle s_1,s_2, s_3 | s_1s_2s_3 \rangle$. So it is enough to lift two generators. 
If $l=2$ and $j_{\EE}$ generic then we can lift the generator  to a set of minimal lifts. With $l=1$ we get all set with one choice of a non minimal lift see also \cite{BT03}.

Taking this into account consider first  the graph $[2,2,2]$ then the associated conjugacy class of modular subgroups $\bar{\Gamma}$  gives rise to two lifts of the form $\Gamma \stackrel{1:1}{\to} \bar{\Gamma}$. Indeed up to symmetry we have one lift that is non minimal at only one point of $B_{\bar{\Gamma}}$. The second  is the lift that it is non minimal at all points of $B_{\bar{\Gamma}}$. Moreover,  we have to add  the unique  lift of the form $\Gamma \stackrel{2:1}{\to} \bar{\Gamma}$. In total we have $3$ possible lifts.  

Next, we consider the  graph $[4,1,1]$ then the associated conjugacy class of modular subgroups $\bar{\Gamma}$  gives rise to three lifts of the form $\Gamma \stackrel{1:1}{\to} \bar{\Gamma}$. Indeed by symmetry we can choose the lifts of parabolic elements corresponding to the face within the loops -- call them $s_1$ and $s_3$ -- and for the third face $s_2$. Here we have two possible choices lifting $s_1$ and $s_2$ or lifting $s_1$ and $s_3$, the other choices being equivalent under $\ZZ/2\ZZ$ symmetry. Now, we can choose in the first case a non-minimal lift for either $s_1$ or $s_2$ and we have two inequivalent lifts. Moreover, we have a single choice up to symmetry in case of lifting $s_1$ and $s_3$.  Finally, we have one more lift of the form $\Gamma \stackrel{2:1}{\to} \bar{\Gamma}$. So we have further $4$ lifts. 

All in all the torsion free case give rise to $7$ possible lifts.   

Secondly, we consider torsion subgroups which contain a 2-torsion element. By Lemma \ref{lem_with2tors} we have a unique lift of a subgroup if there is a 2-torsion element. This translate to a unique lift in case of graphs with a univalent black vertex. Hence the second and fourth graph in Proposition \ref{prop_ix6Tors}  give rise to a unique lift. So this adds $2$ more lifts. 

Finally, we consider torsion groups without 2-torsion elements. In this case we are considering the first and the third graph in Proposition \ref{prop_ix6Tors}.  We always have the unique $\Gamma \stackrel{2:1}{\rightarrow} \bar{\Gamma}$ cover, therefore  we have $2$ more lifts. By Lemma \ref{lem_split} we have a section of the sequence \eqref{eq_seqGG}, and we have to lift the following groups, corresponding respectively to the first and third graph in Proposition \ref{prop_ix6Tors}
\[
H_1 \cong \langle s_1, s_2, t_1| t_1^3, s_1s_2t_1 \rangle \quad H_2 \cong \langle t_1, t_2, s_1| s_1t_1t_2, t_1^3,t^3_2 \rangle.
\]
Let us consider $H_1$, as an example, then we have the following possibilities for the lifts
\begin{enumerate}
\item $s_1$ and $t_1$ minimal lift, and $s_2$ non-minimal;
\item $s_1$ minimal, $t_1$ non minimal and $s_2$ minimal;
\item $s_1$ non minimal, $t_1$  minimal and $s_2$ minimal;
\item $s_1$ non minimal, $t_1$ non minimal and $s_2$ non minimal.
\end{enumerate}
 
One sees that cases (1) and (3) have been already considered (indeed if the lift of $t_1$ is minimal then we are in the case $\Gamma\stackrel{2:1}{\longrightarrow} \bar{\Gamma}$ ), hence we have only two new lifts which are inequivalent (2) and (4).

In a similar way one proceeds for $H_2$ which yields only one lift. In total we have $2+2+1=5$ lifts of torsion subgroups $\bar{\Gamma}$ without 2-torsion elements.

In total we have $7+2+5=14$ possible lifts. 
\end{proof}

\subsection{Index 12 case} This case can be done by looking at the graphs too.

\begin{prop}\label{prop_ix12Tors} The cardinality of the classes of subgroups $\bar{\Gamma} \in \PSL(2,\ZZ)$ laying in  $K_{12}$ is $28$. They are classified according to symmetry of the corresponding torsion free subgroup and $e_2,e_3$ in Table
\ref{K12}, columns 3 -- 7.
\end{prop}

\FloatBarrier
\begin{table}[!h]
\begin{tabular}{|c|c||c|c|c|c|c|c|c|}\hline
 Cases & Sym &TF & $e_2>0$ & $e_2=0,$ & $e_2=0,$  & $e_2=0,$  & 1:1 lifts& 2:1 lifts \\
& & & & $e_3=1$ & $e_3=2$  & $e_3=3$  & &  \\

  \hline
[3,3,3,3] & $A_4$ & 1 & - & - & - & - &  2 & 1 
\\ \hline
[4,4,2,2] & $\ZZ/2\times\ZZ/2$ & 1 & - & - & - & - &  4 & 1 
\\ \hline
[5,5,1,1] & $\ZZ/2$ & 1 &  3 &  1 &  1 & - &  5+0+3+1 & 6 
\\ \hline
[6,3,2,1] & $\{1\}$ & 1 & 1  & 1 & - & - &  7+0+3 & 3 \\
\hline
[8,2,1,1] & $\ZZ/2$ & 1 & 3 & 1 & 1 & - &  5+0+3+1 & 6 \\
\hline
[9,1,1,1] & $\ZZ/3$ & 1 & 7 & 1 & 1 & 1 &  3+0+3+1+0 & 11 \\
\hline
\hline
TOT  & & 6 & 14 & 4 & 3 & 1 &  41 & 28 \\
\hline
\end{tabular}
\caption{Index 12} \label{K12}
\end{table}

\begin{proof} The torsion free subgroups are already given by their dessins in Example \ref{ex_index12}.  Only the graph $[9,1,1,1]$,  $[8,2,1,1]$,  $[6,3,2,1]$, $[5,5,1,1]$ have loops, and therefore can be subjected to substitutions. The following dessins are thus
obtained and -- properly counted -- give the numbers in table \ref{K12}.

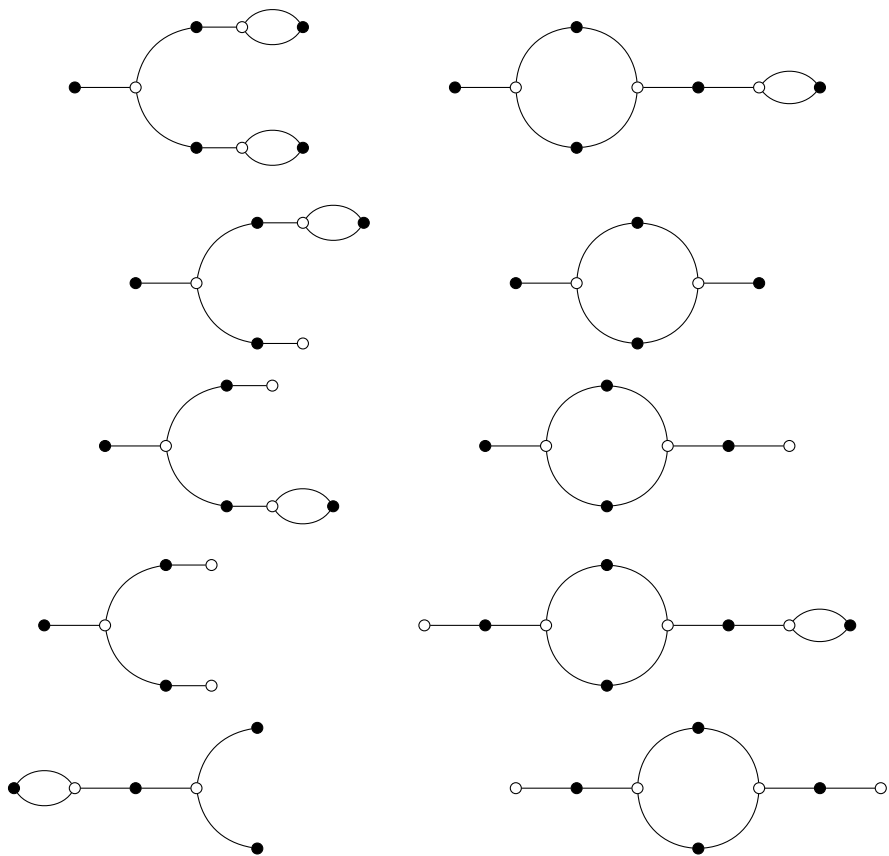
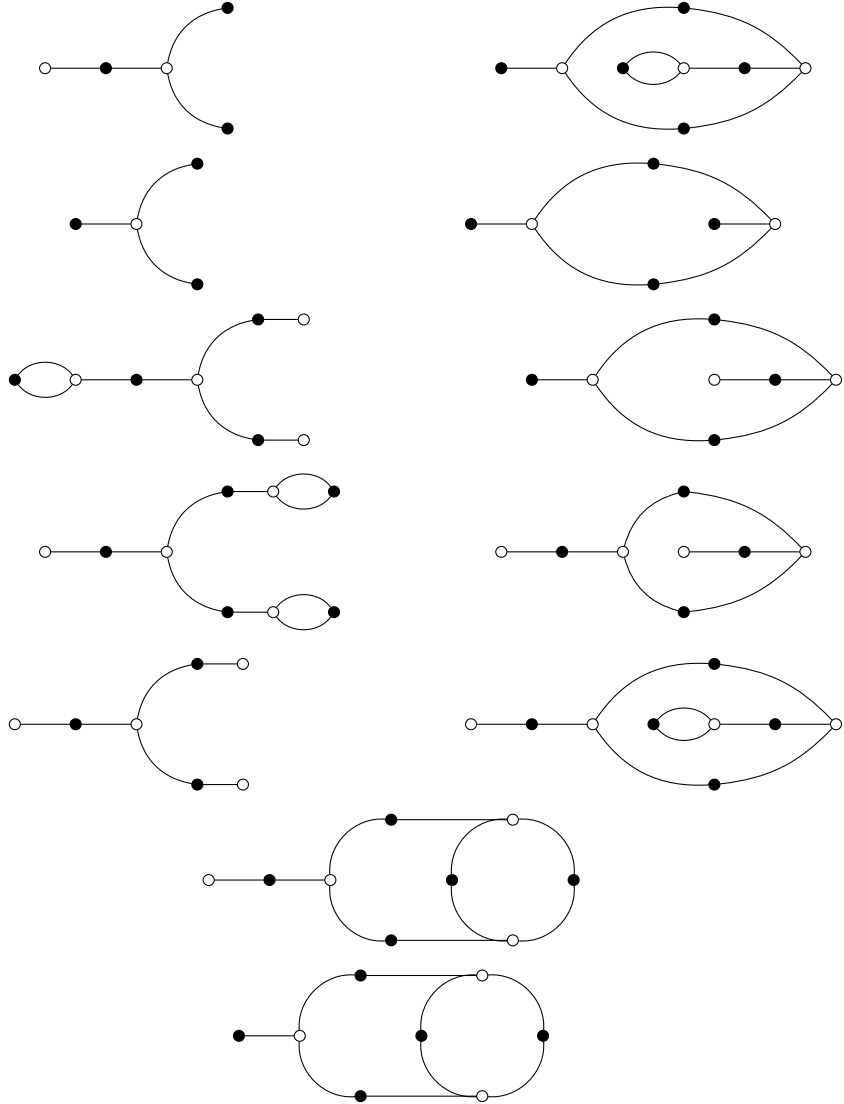
\begin{center}
%
\end{center}


\end{proof}

\begin{prop} The cardinality of $\widetilde{K}_{12}$ of classes of subgroups in $\SL(2,\ZZ)$ is $69$. 
\end{prop} 

\begin{proof} We know that for each of the $28$ groups $\bar{\Gamma}$ in $K_{12}$, we have the unique lift of the form $\Gamma{ \stackrel{2:1}{\rightarrow}}\bar{\Gamma}$. Moreover, we know that if we consider a group with a two torsion elements there is only this lift. Therefore, we have to look at the possible lifts of the form $\Gamma \stackrel{1:1}{\rightarrow} \bar{\Gamma}$ in case 
$\bar{\Gamma}$ is  torsion free or $\bar{\Gamma}$ has $3$ torsion but no $2$ torsion.  We look at the graph in the proof of  Proposition \ref{prop_ix12Tors}  for the latter ones and in Example \ref{ex_index12}  for the former ones.

Let us look first at torsion subgroups. By inspection of the graphs in the proof of Proposition \ref{prop_ix12Tors}  we see that the possibilities for dessins with only monovalent white vertices and no monovalent black vertices are with one, two or three such white vertices.  We have 8 such dessins in total. 
\begin{enumerate}
\item  If we have three white vertices then the $3$ torsion elements must be lifted to $6$ torsion elements. Therefore, we have only the possibility of $\Gamma \stackrel{2:1}{\rightarrow} \bar{\Gamma}$ lift that we have already counted. There is only one subgroup with three $3$-torsion points and this yields no $\Gamma \stackrel{1:1}{\rightarrow} \bar{\Gamma}$ lift.
\item  If we have two white vertices and two parabolic elements, we have only one choice for a lift of the form  $\Gamma \stackrel{1:1}{\rightarrow} \bar{\Gamma}$.  There are $3$ subgroups with two $3$-torsion points and these yields exactly 3   one to one lifts.
\item If we have only one white vertex this lifts to a $3$-torsion element. The parabolic elements are three and for each of these we have a choice of a lift to an elements yielding a $*-$fiber in $\EE$, since the corresponding dessins have no symmetry.  There are $4$  subgroup with only one $3$-torsion point and these yields exactly 12  lifts of the form  $\Gamma \stackrel{1:1}{\rightarrow} \bar{\Gamma}$.
\end{enumerate}
We look now at the $6$ torsion free subgroups.  The groups that we have to lift have a presentation
\[
H:=\langle s_1,s_2,s_3,s_4| s_1s_2s_3s_4\rangle.
\]
It is enough to lift just three parabolic elements, thus  we have $8$ possible lifts. More precisely we have $1$ possibility with all lift to a minimal element, $6$ possibilities with two non minimal and two minimal lift and finally a last possibility with all  lift to a non minimal element.  We have to exclude the case this last case in which all the $s_i$'s are not  minimal lifts, since it leads to an elliptic fibration which is not a K3. Now, remember that we have to pay attention to the symmetry of the graphs. We have $5$ possible lifts if we have a $\ZZ/2$ symmetry and  $3$ possible lifts if we have a $\ZZ/3$ symmetry. By looking at the graphs with loops we have one graph with  $\ZZ/3$ symmetry two with $\ZZ/2$ symmetry and one with no symmetry. 

Looking at the graphs with no loops we have one with $A_4$ symmetry and this gives only two possible lift.  Finally we have a graph with $\ZZ/2 \times \ZZ/2$ symmetry, this yields $3$ possible lifts.

This conclude the proof and we can report all these information in the Table \ref{K12},
columns 8 and 9.
\end{proof}


\subsection{Index 18 case}

\begin{prop}\label{prop_ix18Tors} The cardinality of the classes of subgroups $\bar{\Gamma} \in \PSL(2,\ZZ)$ laying in  $K_{18}$ is $232$.
They are classified according to symmetry of the corresponding torsion free subgroup and $e_2,e_3$ in Table
\ref{K18}.
\end{prop}

\FloatBarrier
\begin{table}[h!]
\begin{tabular}{|ll|c|| c|c|c|c|c|c|}\hline
Cases && Sym &TF & $e_2>0$ & $e_2=0$ & $e_2=0$   \\
& & & & & $e_3=1$ & $e_3>1$    \\
  \hline
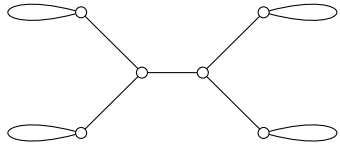
[14,1,1,1,1] & &  $\ZZ/2$  &1 & 35 & 2 & 7    \\
\hline
[13,2,1,1,1] &  [10,5,1,1,1] & $\{1\}$  & 3(1) & 3(19) & 3(3) & 3(4)   \\

 [9,6,1,1,1] & & &  &     &   &           \\
\hline
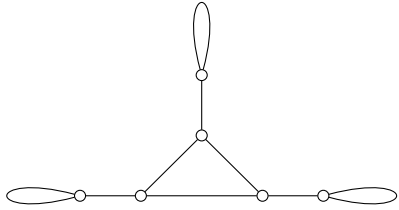
[12,3,1,1,1] & & $\ZZ/3$  & 1 & 7 & 1 & 2   \\
\hline
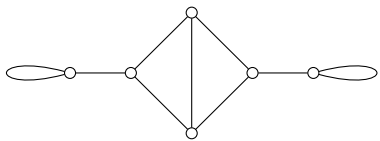
[12,2,2,1,1]& [10,3,3,1,1] & $\ZZ/2$  & 6(1) & 6(3) & 6(1) & 6(1)   \\

[7,7,2,1,1]A &  [7,7,2,1,1]B&  &   &   &   &        \\

[6,6,4,1,1] & [6,5,5,1,1] & &   &   &   &       \\
\hline
[11,3,2,1,1] & [10,4,2,1,1] &  $\{1\}$  & 4(1) & 4(5) & 4(2) & 4(1)   \\

 [9,5,2,1,1] &   [7,6,3,1,1] &  &   &   &   &        \\
\hline
[10,3,2,2,1] & [8,5,2,2,1] & $\{1\}$ & 6(1) & 6(1) & 6(1) & 0  \\

[8,4,3,2,1] &  [7,5,3,2,1]  &  &   &  &   &        \\

[7,4,3,3,1] & [6,5,4,2,1] & & & & &   \\
\hline
[8,3,3,2,2] & [6,4,4,2,2] &  $\ZZ/2$ & 3(1) & 0 & 0 & 0  \\

[5,5,3,3,2] &  &  &  &   &   &        \\
\hline
[6,6,2,2,2] & [4,4,4,3,3] &   $S_3$ & 2(1) & 0 & 0 & 0   \\
\hline
\hline
total: 232 & & & 26 & 143 & 32  & 31    \\
\hline
\end{tabular}
\caption{Index 18} \label{K18}
\end{table}

\begin{proof} The idea is the same as for Propositions \ref{prop_ix6Tors} and \ref{prop_ix12Tors}. Here we have to analyze the dessins given in Example \ref{ex_index18} with loops and proceed with all possible substitutions of loops by univalent white or black vertices up to symmetry. 
Clearly the number of graphs become too large to draw them all, so we need to make some considerations in order to count them. 

For example there is only one graph with $4$ loops, $[14,1,1,1,1]$, and it has a $\ZZ/2$ symmetry. \\ 
First, we want to replace $e_3>1$ loops with as many white vertices. We get then $ \binom{4}{2}+ \binom{4}{3}+1=11$ dessins. But taking the $\ZZ/2$ symmetry into account, we have only $4$ distinct dessins for two white vertices and only $2$ with three white vertices. Therefore we have $7$ distinct dessins in total, see Table \ref{K18} line 1,  column 6. \\ 
Secondly,  we want to replace only $e_3=1$ loop with one white vertex than we have only $2$ possibilities up to symmetry,  see Table \ref{K18} line 1 column 5.  \\ 
Finally, we consider the case with $e_2>0$ black vertices. Here there are many more possibilities. If we replace one loop with a black vertex then the other three have three possibilities each, black, white or loop. Of course we have to take the $\ZZ/2$ symmetry into account. A moment of thought gives that  there is only one orbit if we have four black vertices, $4$ with three black vertices, $14$ with two and $16$ with one. In total there are 35 dessins, see Table \ref{K18} line 1, column 4. 
\\
The same is done for the other graphs. To obtain table  \ref{K18} one has to consider the following. 
\begin{enumerate}
\item Dessins with at least two loops and at least two white vertices substitutions but no black ones (column 6 Tab. \ref{K18}) 
\item Dessins with one or more loops and only one white vertex substitution but no black ones (column 5 Tab. \ref{K18}) 
\item Dessins with at least one black vertex substitution  (column 4 Tab.  \ref{K18}) 
\end{enumerate}

 The proof is completed summing up the numbers in the last row of the column $3$, $4$, $5$ and $6$ of Table \ref{K18}. 
\end{proof}

\begin{prop}  The cardinality of $\widetilde{K}_{18}$ of classes of subgroups in $\SL(2,\ZZ)$ is  $366$.
\end{prop}

\begin{proof} The proof consist again in adding fibre cardinalities. 
Let us keep in mind, that for each subgroup we have a unique $\Gamma \stackrel{2:1}{\to} \bar{\Gamma}$ lift and proceed by addressing each column of Table \ref{K18} separately.

In case with a $2$-torsion element (column 4), the ${2:1}$ lift is the only possible by Lemma \ref{lem_with2tors}, they contribute 143 to the total number.

For all other modular subgroups in $K_{18}$ corresponding K3 surfaces must
be of kind $S_{\bar\Gamma}$ according to Proposition \ref{prop_jfactor1824}.

By the Euler number formula, precisely one fibre is a $*$-fibre then.

Let us now consider the possible lifts of torsion free group. To count them we have to take into account the symmetry of the graphs, they are recalled in Tab. \ref{K18}. Indeed, for each group we have as many  possible $1:1$ lifts as the orbits of the faces of the corresponding graph under the action of the symmetry group. Thus in the absence of symmetry we have $5$ inequivalent lifts, we have $3$ under a $\ZZ/2$ action, $3$ for a $\ZZ/3$ action and finally $2$ for the $D_3$ action. Together the lifts of torsion free groups (column 3) contribute $5(13)+3(11)+2(2)+26=128$ to the total number. 

We get to the groups corresponding to the 5th column of Tab. \ref{K18}.
In this case we have a unique $1:1$ lift, corresponding to lifting the $3$-torsion
generator to a $3$-torsion element in $\SL(2,\ZZ)$ which fixes the choice of the
$*$-fibre. Together with the ${2:1}$ lift in each case the contribution of column 5 to the total number is $32+32=64$. 

Finally, we see that there is no $1:1$ lift realized by a K3 if the have more than one $3$-torsion element, since lifting more than one $3$-torsion element to a $3$-torsion element imposes more than two $*$-fibres. Therefore the contribution of column 6 to the total number is $31$.

Looking at the Table, we have proven the cardinality to be $128+143+64+31=366$.
\end{proof}

\subsection{Index 24 case}

The number of dessins in $K_{24}^{\tf}$ is $191$, too big to draw them all.
Instead we refer to the list in \cite{HHP21}, another source is given in \cite{BM}.
To count the number of elements in $K_{24}$ we first decompose $K_{24}^{\tf}$ into
subsets which have constant fibre cardinality for the surjection $K_{24}\to K_{24}^{\tf}$.
Again the symmetry groups are of essential importance.

\begin{defin} 
A white vertex has \emph{type $a|b|c$} if the face degrees of adjacent faces are $a$, $b$ and $c$ in positive order with $a\geq b\geq c$ or $a>c>b$.
\end{defin}

\begin{lem} \label{lem_freetype}
For any dessin, the symmetry group $A$ acts freely on white vertices of any given type $a|b|c$,
except possibly in case $3$ divides the order of $A$ and types with $a=b=c$.
\end{lem}

\begin{proof}
An element in $A$ is the identity if it fixes any edge. If a white vertex is fixed by a non-trivial element $\alpha$
of $A$, then the adjacent faces are in one orbit of $A$, hence their degrees must be equal.
Moreover $\alpha$ has order $3$, since $\alpha^3$ fixes an edge, so the order of $A$ must be a multiple of $3$.
\end{proof}

Analogous claims can be made about black vertices having white neighbours of different types
or white vertices with white neighbours of given types.

\begin{prop}\label{prop_sym24} 
All dessins corresponding to modular subgroups in $K_{24}^\tf$ with at least one loop 
have trivial symmetry group except for the dessins collected in Table \ref{Tableindex24Sym}, where they are listed according to their number of loops and
their symmetry groups.

\FloatBarrier
\begin{table}[h!]
\begin{tabular}{|lll|c|c|}\hline
& Cases && Sym &  $\sharp$-graphs \\
\hline
$[6,6,5,5,1,1]\, A,B$,& $[6,6,6,4,1,1],$ & $[7,7,4,4,1,1] \, A,B$ && \\
$ [8,8,3,3,1,1]\, A,B,C,$ & $[8,8,4,2,1,1],$ & $[9,9,2,2,1,1]\, A,B,C,D$ & $\ZZ/2$ & $18$ \\
$ [10,5,5,2,1,1] \,A,B,$ & $[12,4,3,3,1,1]\, A, $ & $[12,6,2,2,1,1] \, A$ & & \\
$[16,2,2,2,1,1],$ & & & & \\
\hline
$[10,10,1,1,1,1] \, B,C,$  & $[18,2,1,1,1,1] \, A$ & & $\ZZ/2$ & $3$ \\
\hline
&$[7,7,7,1,1,1]$ && $\ZZ/3$ & $1$ \\
\hline
&$[10,10,1,1,1,1] A$ && $\ZZ/2 \times \ZZ/2$ & $1$ \\
\hline
&$[16,4,1,1,1,1]$ && $\ZZ/4$ & $1$ \\
\hline   
\end{tabular}
\caption{Index 24 Symmetry} \label{Tableindex24Sym}
\end{table}
\end{prop}

\begin{proof}
We go through the list of dessins in \cite{HHP21} with at least one loop and for each we
determine $d$, the gcd of the cardinalities of white vertices of any given type $a|b|c$
except those with $a=b=c$. It is particularly easy to see $d=1$ in case of one loop, which
corresponds to a unique white vertex of type $a|a|1$ with $a>1$, or more generally for
any dessin with a unique white vertex of a type as given in Lemma \ref{lem_freetype}.
We find only the following mutually exclusive cases:
\begin{enumerate}
\item[$d=4$]
$[10,10,1,1,1,1] A,B$ and $[16,4,1,1,1,1]$.
\item[$d=3$]
$[7,7,7,1,1,1]$, the cardinality of vertices of type $7|7|7$ is two, but it must be discarded
when forming the gcd.
\item[$d=2$]
the dessins with two loops occurring in Table \ref{Tableindex24Sym} and 
$[12,6,2,2,1,1]B$,
$[14,3,3,2,1,1]B$, and four dessins with four loops
$[10,10,1,1,1,1]C$,
$[18,2,1,1,1,1] A,B,C$.
\item[$d=1$]
all remaining dessins with at least one loop.
\end{enumerate}
It remains to determine the group of symmetries in the $28$ cases with $d>1$. 
Symmetries can be detected by inspection of the dessins given in \cite{HHP21}, or checking
each permutation of the white vertices preserving types and without fix points except when $a=b=c$,
whether it is induced by a symmetry of the dessin.
It then turns out, that $d$ is the order of $A$ except in the cases
$[12,6,2,2,1,1]B$,
$[14,3,3,2,1,1]B$, 
$[18,2,1,1,1,1] B,C$, where $d=2$ but $A$ is trivial, and
$[10,10,1,1,1,1] B$, where $d=4$ but $A$ has order $2$.
In $d=4=|A|$ the symmetry group $\ZZ/4$ can be seen on the dessin, this is how the claim
can be established.
\end{proof}

\begin{prop} \label{prop_tot24} The cardinality of conjugacy classes of subgroups $\bar{\Gamma} \in \PSL(2,\ZZ)$ belonging to  $K_{24}$ is  $2962$.
\end{prop}

\begin{proof}  We rely merely on a combinatorial computation that is summarized in 
Table \ref{Tableindex24}. 

\begin{table}[!h]
\begin{tabular}{|c|c|c|c||c|}\hline
$\sharp$-loops & Sym & $\sharp$-graphs & Mult. Factor & tot \\
\hline
0 & $arb.$ & 20 & 1 & 20 \\  
1 & $\{1\}$ & 45 & 3 & 135 \\
2 & $\{1\}$ & 53 & 9 & 477 \\
2 & $\ZZ/2$ & 18 & 6 & 108 \\
3 & $\{1\}$ & 39 & 27 & 1053  \\
3 & $\ZZ/3$ & 1 & 11 & 11  \\
4 & $\{1\}$ & 9 & 81 & 729 \\
4 & $\ZZ/2$ & 3 & 45 &  135\\
4 & $(\ZZ/2)^2$ & 1 & 27 & 27\\
4& $\ZZ/4$ & 1 & 24 & 24 \\
5 & $\{1\}$ & 1 & 243 & 243 \\        
\hline 
\hline          
 &  &  191 & & 2962  \\  
\hline   
\end{tabular}
\caption{Index 24} \label{Tableindex24}
\end{table}

The first three columns address the dessins in $K_{24}^{\tf}$ classified by the number
of loops and -- for positive number of loops -- by symmetry groups. The fibres of
$K_{24}\to K_{24}^{\tf}$ have finite cardinality which is $1$ in the absence of loops
and equal to the number of orbits for the induced action of the symmetry group
on the maps from the set of loops to the set of substitutions including the possibility of
a loop
\begin{center}
\begin{tikzpicture}[-latex, node distance = 2cm]
\node(o){};
\node(o3)[left of= o]{};
\node(o4)[left of= o3]{};
\node[circle,draw, fill=black, scale=0.4](aaa)[right of= o4]{};
\node[circle,draw, fill=black, scale=0.4](aa)[right of= o3]{};
\node[circle, draw, scale=0.4](bb)[right of= aa]{};
\node[, scale=0.4](o2)[right of= o]{};
\node[circle,draw, fill=black, scale=0.4](a)[right of= o2]{};
\node[circle, draw, scale=0.4](b)[right of= a]{};
\node[circle, draw, fill=black, scale=0.4](c)[right of= b]{};
\path[-](o4) edge (aaa);
\path[-](o3) edge (aa);
\path[-](aa) edge (bb);
\path[-](a) edge (b);
\path[-](o2) edge (a);
\path [-](b) edge [bend left =55](c);
\path [-](b) edge [bend right =55](c);
\end{tikzpicture}
\end{center}
As the action is free on the loops it is an elementary task, 
e.g.\ with Burnside's Lemma, to determine this cardinality, which we put in the fourth
column. Finally we multiply by this factor the number of dessins in each class
to get the number of corresponding dessins in $K_{24}$ in the last column.
The sum of these entries is the cardinality of $K_{24}$.
\end{proof}

\begin{prop}\label{prop_ktilde24}  
The cardinality of $\widetilde{K}_{24}$ of classes of subgroups in $\SL(2,\ZZ)$ is  $2962$.
\end{prop}

\begin{proof} 
We have to see that the map from $\widetilde{K}_{24}$ to $K_{24}$ is a bijection. 
It is onto by Thm. \ref{theo_main42}.
Bijectivity is then automatic at all modular subgroups with $e_2>0$ since such groups
are the image only of its preimage in $SL_2(\ZZ)$.
For all modular subgroups Corollary \ref{cor_Euler} shows, that only for the unique choice with no
$*$-fibres we get a $K3$, which is thus the unique K3 with modular monodromy $\bar\Gamma$
and $j$-factor $l=1$.
By Prop. \ref{prop_jfactor1824} there is no further K3 with $j$-factor $l>1$ in the remaining cases
with $e_2=0$.
Hence bijectivity is shown and we take the number from Proposition \ref{prop_tot24}.
\end{proof}


\end{document}